\documentclass[12pt]{article}

\usepackage{amsmath}
\usepackage{amsthm}
\usepackage{amssymb}
\usepackage{cite}
\usepackage{url}
\usepackage[multiple]{footmisc}
\usepackage{graphicx}
\usepackage{epstopdf}
\usepackage{chngcntr}
\usepackage{enumitem}
\usepackage{hhline}
\usepackage{array}
\usepackage{hyperref}
\usepackage{color}
\usepackage{array}
\usepackage{multirow}

\hypersetup{colorlinks=false}

\if@twoside
\else 
\fi

\numberwithin{equation}{section}
\counterwithin{figure}{section}
\counterwithin{table}{section}

\theoremstyle{plain}% default
\newtheorem{theorem}{Theorem}[section]

\newtheorem{proposition}[theorem]{Proposition}

\newtheorem{corollary}[theorem]{Corollary}

\theoremstyle{definition}
\newtheorem{definition}{Definition}[section]

\newtheorem{example}{Example}[section]

\theoremstyle{remark}
\newtheorem{remark}{Remark}[section]

\newcommand{\norm}[1]{\left\|#1\right\|}
\newcommand{\newnorm}[1]{\left|\left|\left|#1\right|\right|\right|}
\newcommand{\abs}[1]{\left\vert#1\right\vert}
\newcommand{\spr}[1]{\left\langle\,#1\,\right\rangle}
\newcommand{\kl}[1]{\left(#1\right)}
\newcommand{\Kl}[1]{\left\{#1\right\}}

\definecolor{aog}{rgb}{0.0, 0.5, 0.0}

\newcommand{\R}{\mathbb{R}} 
\newcommand{\C}{\mathbb{C}}
\newcommand{\N}{\mathbb{N}}
\newcommand{\Z}{\mathbb{Z}}

\newcommand{\xkd}{x_k^\delta}

\newcommand{\xkpd}{x_{k+1}^\delta}

\newcommand{\yd}{y^{\delta}}

\newcommand{\eps}{\varepsilon}

\newcommand{\RN}{{\R^N}}
\newcommand{\ZN}{{\Z^N}}

\newcommand{\F}{\mathcal{F}}
\newcommand{\FI}{\mathcal{F}^{-1}}

\newcommand{\CmO}{{C^m(\Omega)}}
\newcommand{\CiO}{{C^\infty(\Omega)}}
\newcommand{\CizO}{{C^\infty_0(\Omega)}}

\newcommand{\LtO}{{L_2(\Omega)}}
\newcommand{\LtOp}{{L_2(\Omega')}}
\newcommand{\LtR}{{L_2(\R)}}
\newcommand{\LtRN}{{L_2(\R^N)}}
\newcommand{\LiR}{{L_\infty(\R)}}
\newcommand{\LtTN}{{L_2(\mathbb{T}^N)}}

\newcommand{\HsO}{{H^s(\Omega)}}
\newcommand{\HsR}{{H^s(\R)}}
\newcommand{\HsRN}{{H^s(\R^N)}}

\newcommand{\HmO}{{H^m(\Omega)}}
\newcommand{\HmzO}{{H_0^m(\Omega)}}

\newcommand{\HsTN}{{H^s(\mathbb{T}^N)}}

\newcommand{\Xs}{{X_s}}

\newcommand{\vphi}{\varphi}
\newcommand{\laplace}{\Delta}

\newcommand{\M}{\mathcal{M}}

\newcommand{\Mv}{\textbf{M}}

\newcommand{\nv}{\textbf{n}}
\newcommand{\zv}{\textbf{z}}
\newcommand{\uv}{\textbf{u}}
\newcommand{\vv}{\textbf{v}}
\newcommand{\ydv}{\textbf{y}^\delta}
\newcommand{\Am}{\textbf{A}}

\newcommand{\Esmn}{E_{m,n}^{s}}

\newcommand{\HXm}{\textbf{H}_{X_m}}
\newcommand{\HYn}{\textbf{H}_{Y_n}}
\newcommand{\Mmn}{\textbf{M}_{m,n}}

\title{Characterizations of Adjoint Sobolev Embedding Operators with Applications in Inverse Problems}
\author{
Simon Hubmer\footnote{Johann Radon Institute Linz, Altenbergerstra{\ss}e 69, A-4040 Linz, Austria, (simon.hubmer@ricam.oeaw.ac.at), Corresponding author.} ,
Ekaterina Sherina\footnote{University of Vienna, Faculty of Mathematics, Oskar Morgenstern-Platz 1, 1090 Vienna, Austria (ekaterina.sherina@univie.ac.at)} ,
Ronny Ramlau\footnote{Johannes Kepler University Linz, Institute of Industrial Mathematics, Altenbergerstra{\ss}e 69, A-4040 Linz, Austria, (ronny.ramlau@jku.at)} \footnote{Johann Radon Institute Linz, Altenbergerstra{\ss}e 69, A-4040 Linz, Austria, (ronny.ramlau@ricam.oeaw.ac.at)}
}

\begin{document}

% Include the title
\maketitle

% Abstract
\begin{abstract}

We consider the Sobolev embedding operator $E_s : \HsO \to \LtO$ and its role in the solution of inverse problems. In particular, we collect various properties and investigate different characterizations of its adjoint operator $E_s^*$, which is a common component in both iterative and variational regularization methods. These include variational representations and connections to boundary value problems, Fourier and wavelet representations, as well as connections to spatial filters. Moreover, we consider characterizations in terms of Fourier series, singular value decompositions and frame decompositions, as well as representations in finite dimensional settings. While many of these results are already known to researchers from different fields, a detailed and general overview or reference work containing rigorous mathematical proofs is still missing. Hence, in this paper we aim to fill this gap by collecting, introducing and generalizing a large number of characterizations of $E_s^*$ and discuss their use in regularization methods for solving inverse problems. The resulting compilation can serve both as a reference as well as a useful guide for its efficient numerical implementation in practice.

\end{abstract}

% % % % % % % % % % % % %
% Start of the sections %
% % % % % % % % % % % % %

% % % % % % % % % % % % % %
% Section - Introduction  %
% % % % % % % % % % % % % %
\section{Introduction}

In this paper, we consider the Sobolev embedding operator
    \begin{equation}\label{embedding}
        E_s \, : \, \HsO \to \LtO \,,
        \qquad
        u \mapsto E_s u := u \,,
    \end{equation}
where $\HsO$ denotes the Sobolev space of order $s \in \R$, and $\Omega \subseteq \R^N$ for some $N \in \N$. The embedding operator $E_s$ not only plays a role in the proper definition of inverse problems, but is also crucial in their solution. Especially its adjoint operator $E_s^*$ commonly appears in regularization methods, and thus close attention needs to be paid to its properties and efficient implementation. Please note that while the implementation of the embedding operator $E_s$ itself is trivial, the evaluation of its adjoint $E_s^*$ is not. 

Due to its ubiquitous appearance in inverse problems, researchers have used a number of different approaches for dealing with the adjoint embedding operator $E_s^*$. These are typically either based on proper discretizations of the underlying problems, or on characterizations of $E_s^*$ in terms of certain variational or boundary value problems; see e.g.~\cite{Neubauer_1989,Ramlau_Teschke_2004_1,Engl_Hanke_Neubauer_1996,Kaltenbacher_Neubauer_Scherzer_2008,Neubauer_2000}. Implicitly, the connection between the (adjoint) Sobolev embedding operator and differential equations is also present in classic reference works such as \cite{Evans_1998,Adams_1970,McLean_2000,Necas_2011,Gilbarg_Trudinger_1998,Agmon_1965}. Furthermore, characterizations of $E_s^*$ in terms of Fourier series as well as Fourier and wavelet transforms are explicitly considered in a one-dimensional setting in \cite{Ramlau_2008,Ramlau_Teschke_2004_1}, and are implicitly present e.g.\ in \cite{Adams_Fournier_1977,Daubechies_1992,Meyer_1993, McLean_2000, Aronszajn_Smith_1961}. Unfortunately, these approaches for characterizing the adjoint embedding operator $E_s^*$ are scattered throughout the literature, and their description is often very brief, restricted to simple cases, or given without explicitly specifying underlying assumptions or providing proper mathematical proofs.

Hence, with this paper we want to bridge this gap, collecting, introducing and generalizing a large number of different approaches to and characterizations of the adjoint embedding operator $E_s^*$. With this, we aim to establish a common reference point for both researchers and students in Inverse Problems and other fields. Among the considered topics are characterizations of $E_s^*$ as the solution of certain boundary value problems, representations in terms of Fourier and wavelet transforms, as well as characterizations via Bessel potentials and spatial filters. Furthermore, we consider representations in terms of Fourier series, singular value decompositions and frame decompositions, as well as different representations in finite dimensional settings. Whenever the basic ideas of the presented approaches could be traced back to previous publications, we provided proper references, and aimed to generalize the corresponding approaches as far as possible, for example by extending them to higher dimensions, by explicitly stating otherwise implicit or overlooked assumptions, and by providing concise proofs for all of them.

The outline of this paper is as follows: In Section~\ref{sect_Sobolev_Spaces_Embeddings}, we review some basic definitions and results on Sobolev spaces that are used throughout the paper. In Section~\ref{sect_PDEs}, we present characterizations via boundary value problems, and in Sections~\ref{sect_fourier_transform} we consider representations via Fourier transforms, which are also the starting point for representations via spatial filters given in Section~\ref{sect_filters}. This is followed by wavelet characterizations in Section~\ref{sect_wavelets}, via Fourier series in Section~\ref{sect_series}, and via singular value and frame decompositions in Section~\ref{sect_SVD_Frames}. Finally, we consider representations in finite dimensional settings in Section~\ref{sect_discrete}, followed by applications of the adjoint embedding operator in inverse problems and a numerical example in Section~\ref{sect_application}, as well as a short conclusion in Section~\ref{sect_conclusion}.

% % % % % % % % % % % % % % % % % % % % % % 
% Section - Sobolev Spaces and Embeddings %
% % % % % % % % % % % % % % % % % % % % % %
\section{Sobolev Spaces and Embeddings}\label{sect_Sobolev_Spaces_Embeddings}

In this section, we recall the definition and some general facts about Sobolev spaces and the Sobolev embedding operator, which are collected from \cite{Evans_1998,Adams_Fournier_2003,Adams_1970,Adams_Fournier_1977,Neubauer_1988,McLean_2000,Necas_2011}.

% Subsection - Integer order Sobolev spaces
\subsection{Sobolev spaces of integer order}

Throughout this paper, we use the following standard notations and assumptions:
    \begin{itemize}
        \item Let $N \in \N$ and let the domain $\Omega \subseteq \RN$ be nonempty and open.
        \item For all $m \in \N$, let $\CmO$ denote the vector space of all continuous functions $u : \Omega \to \C$ whose partial derivatives up to order $m$ are also continuous on $\Omega$.
        \item Let $\CiO := \bigcap_{m=0}^\infty \CmO$ be the space of infinitely differentiable functions, and let $\CizO$ be the subspace of all functions in $\CiO$ with compact support in $\Omega$.
        \item Let $\alpha = \kl{\alpha_1\,, \dots \,, \alpha_N}$ denote a multiindex, i.e., an $N$-tuple of non-negative integers, and let $\abs{\alpha} := \alpha_1 + \dots + \alpha_N$ denote the order of $\alpha$.   
    \end{itemize}
\vspace{0pt}

\begin{definition}
The Lebesgue space $\LtO$ is defined as
    \begin{equation*}
        \LtO := \Kl{ u : \Omega \to \R \, \vert \, u \,\, \text{is measurable and} \, \norm{u}_\LtO < \infty   } \,,
    \end{equation*}
where the norm
    \begin{equation*}
        \norm{u}_\LtO := \kl{\int_\Omega \abs{u(x)}^2 \, dx}^{1/2} \,.
    \end{equation*}
Hereby, we identify functions in $\LtO$ which are equal almost everywhere in $\Omega$. 
\end{definition}

The space $\LtO$ is a separable Hilbert space when equipped with the inner product
    \begin{equation*}
        \spr{u,v}_\LtO := \int_\Omega u(x) \overline{v(x)} \, dx \,.
    \end{equation*}
In order to define Sobolev spaces, we first need to introduce the derivatives
    \begin{equation*}
        D^\alpha u := \frac{\partial^{\abs{\alpha}}}{\partial x_1 ^{\alpha_1} \, \cdots \, \partial x_N^{\alpha_N}} u 
        =
        \partial_{x_1}^{\alpha_1} \, \cdots \, \partial_{x_N}^{\alpha_N} u \,.
    \end{equation*}
Throughout this paper, derivatives are understood in the weak sense, i.e., if $u$ and $v$ are locally integrable functions, then $v$ is the $\alpha^\text{th}$-weak partial derivative of $u$ if and only if 
    \begin{equation*}
        \spr{u, D^\alpha \phi}_\LtO = (-1)^{\abs{\alpha}}\spr{v , \phi}_\LtO \,,
        \quad 
        \forall \phi \in \CizO \,.
    \end{equation*}
In this case, we write $D^\alpha u = v$. With this, we can now introduce Sobolev spaces in 

\begin{definition}
For any $m \in \N$, the Sobolev space $\HmO$ of order $m$ is defined by
    \begin{equation}\label{def_HmO_Da_full}
        \HmO := \Kl{ u \in \LtO \, \vert \, \text{ $D^\alpha u \in \LtO$ for all $\alpha$ with $0 \leq \abs{\alpha} \leq m$}} \,.
    \end{equation}
On the space $\HmO$ we define the norm
    \begin{equation}\label{def_HmO_Da_full_norm}
        \norm{u}_\HmO := \kl{\sum\limits_{0\leq \abs{\alpha} \leq m} \norm{D^\alpha u }_\LtO^2 }^{1/2} \,.
    \end{equation}
Furthermore, we define $\HmzO$ as the closure of $\CizO$ in the space $\HmO$.
\end{definition}

Note that $H^0(\Omega) = H^0_0(\Omega) = \LtO$. Furthermore, equipped with the inner products
    \begin{equation}\label{def_HmO_Da_full_inner}
        \spr{u,v}_\HmO :=  \sum\limits_{0\leq \abs{\alpha} \leq m} \spr{D^\alpha u , D^\alpha v}_\LtO \,,
    \end{equation}
the Sobolev spaces $\HmO$ are separable Hilbert spaces. Next, consider the seminorms
    \begin{equation}\label{def_seminorm_Hm}
        \abs{u}_\HmO := \kl{ \sum\limits_{\abs{\alpha} = m}  \norm{ D^\alpha u}_\LtO^2   }^{1/2} \,,
    \end{equation}
which in certain situations can be used to define equivalent norms on $\HmO$ and $\HmzO$, influencing the corresponding inner products and in turn also $E_s^*$. For discussing these situations, we first need to introduce two geometric properties on the domain $\Omega$, which are given in the following definition adapted from \cite{Adams_Fournier_1977,Adams_Fournier_2003}.

\begin{definition}\label{def_domains}
For all $x \in \Omega \subseteq \RN$ let $R(x)$ denote the set of all $y \in \Omega$ such that the line segment from $x$ to $y$ lies entirely in $\Omega$. Then $\Omega$ satisfies the \emph{weak cone condition} if and only if there exists a $\delta > 0$ such that for all $x \in \Omega$ there holds
    \begin{equation*}
        \lambda_N\kl{\Kl{y \in R(x) \, \vert \, \abs{y-x}<1}} \geq \delta  \,,
    \end{equation*}
where $\lambda_N$ denotes the Lebesgue measure in $\RN$. Furthermore, the domain $\Omega$ is said to have a \emph{finite width} if it lies between two parallel lines of dimension $(N-1)$.
\end{definition}

Using these geometric properties, we can now state the following result; cf.~\cite{Adams_1970,Adams_Fournier_2003,Necas_2011}.
\begin{proposition}\label{prop_norm_equivalence_Hm}
If the domain $\Omega \subseteq \RN$ satisfies the weak cone condition, then 
    \begin{equation}\label{def_HmO_Da_simple_norm}
        \newnorm{u}_\HmO := \kl{\norm{u}_\LtO^2 + \abs{u}_\HmO^2 }^{1/2} \,,
    \end{equation}
defines a norm on $\HmO$ which is equivalent to $\norm{\cdot}_\HmO$. Similarly, if $\Omega$ has a finite width, then the seminorm $\abs{\cdot}_\HmO$ itself is a norm on $\HmzO$ equivalent to $\norm{\cdot}_\HmO$. 
\end{proposition}
\begin{proof}
The first part of the this proposition essentially follows from \cite[Theorem~2]{Adams_Fournier_1977}; see also \cite{Necas_2011}. The second part was originally shown in \cite{Adams_1970}; see also \cite[Corollary~6.31]{Adams_Fournier_2003}.
\end{proof}

Note that the norm $\newnorm{\cdot}_\HmO$ considered in \eqref{def_HmO_Da_simple_norm} is induced by the inner product 
    \begin{equation}\label{def_HmO_Da_simple_inner}
        \spr{u,v}_\HmO :=  \spr{u,v}_\LtO + \sum\limits_{\abs{\alpha} = m} \spr{D^\alpha u , D^\alpha v}_\LtO \,.
    \end{equation}
Similarly, on $\HmzO$ the equivalent seminorm $\abs{\cdot}_\HmO$ is induced by the inner product
    \begin{equation}\label{def_seminorm_Hm_inner}
        \spr{u,v}_\HmO :=  \sum\limits_{\abs{\alpha} = m} \spr{D^\alpha u , D^\alpha v}_\LtO \,.
    \end{equation}
    
\begin{remark}
Concerning the geometric assumptions in Proposition~\ref{prop_norm_equivalence_Hm}, note that the domain $\Omega$ satisfies the weak cone condition e.g.\ if its boundary is uniformly $C^m$ regular for some $m \geq 2$, or if it satisfies the strong local Lipschitz condition  \cite{Adams_Fournier_2003}. Furthermore, the weak cone condition is also satisfied if $\Omega = \RN$ or some half space in $\RN$. Moreover, note that any bounded domain $\Omega$ also has finite width. 
\end{remark}

% Subsection - Fractional order Sobolev spaces
\subsection{Sobolev spaces of fractional order}

Next, we consider Sobolev spaces $\HsO$ of fractional order, i.e., when $0 \leq s \in \R$ is not necessarily an integer. These can be defined in different ways, for example:
    \begin{enumerate}
        \item Via interpolation theory, as intermediate spaces between $\LtO$ and $H^{\lceil s \rceil}(\Omega)$ \cite{Adams_Fournier_2003, Bergh_Loefstroem_1976}.
        \item Via the Slobodeckij seminorm, then also called Sobolev-Slobodeckij spaces \cite{Adams_Fournier_1977,McLean_2000}.
        \item Via Bessel potentials, then also called Bessel-potential spaces \cite{Adams_Fournier_1977,Evans_1998,McLean_2000}.
    \end{enumerate}
These definitions are equivalent as long as the domain $\Omega$ is sufficiently regular \cite{Adams_Fournier_1977}. In this paper, we restrict ourselves to fractional Sobolev spaces over $\RN$ and focus on the definition via Bessel potentials. For this, we first need to recall the following 

\begin{definition}
For any integrable $u : \RN \to \C$ we define the Fourier transform 
    \begin{equation*}
        (\F u)(\xi) = \int_\RN e^{-2\pi i \xi \cdot x} u(x) \, dx \,,
        \qquad
        \forall \, \xi \in \RN \,,
    \end{equation*}
as well as the inverse Fourier transform
    \begin{equation*}
        (\FI u)(x) = \int_\RN e^{2\pi i x \cdot \xi} u(\xi) \, d\xi \,, 
        \qquad
        \forall \, x \in \RN \,.
    \end{equation*}
\end{definition}

Both operators $\F$ and $\FI$ uniquely determine bounded linear operators
    \begin{equation*}
        \F : \LtRN \to \LtRN \,,
        \qquad
        \FI : \LtRN \to \LtRN \,,
    \end{equation*}
satisfying $\F \FI = \FI \F = I$, cf.~\cite{McLean_2000,Evans_1998}, as well as the convolution property
	\begin{equation}\label{Fourier_convolution}
		\F( u \ast v ) = \F(u) \F(v) \,,
		\qquad
		\text{where}
		\qquad
		(u \ast v)(x) := \int_\RN u(x-y) v(y) \, dy \,.
	\end{equation}
With this, we can now introduce fractional order Sobolev spaces over $\RN$ in

\begin{definition}
For any $0 \leq s \in \R$ the Sobolev space $\HsRN$ of order $s$ is defined by
    \begin{equation}\label{def_HsO_Bessel}
        \HsRN := \Kl{ u \in \LtRN \, \vert \, \FI\kl{\kl{1+4\pi^2 \abs{\cdot}^2}^{s/2} \F u  } \in \LtRN }
    \end{equation}
On the space $\HsRN$ we define the norm
    \begin{equation}\label{def_HsO_Bessel_norm_v1}
        \norm{u}_\HsRN := \norm{ (1+ 4\pi^2\abs{\cdot}^2)^{s/2} \F u  }_\LtRN \,.
    \end{equation}
\end{definition}

The Sobolev spaces $\HsRN$ defined in \eqref{def_HsO_Bessel} are also called Bessel-potential spaces; cf.\ Section~\ref{sect_fourier_transform}. They become Hilbert spaces when equipped with the inner products
    \begin{equation}\label{def_HsO_Bessel_inner_v1}
        \spr{u,v}_\HsRN := \int_\RN \kl{1+ 4\pi^2 \abs{\xi}^2}^s (\F u)(\xi) \overline{(\F v)(\xi)} \, d\xi \,.
    \end{equation}
Note that instead of $4 \pi^2$ any other positive constant could be used in \eqref{def_HsO_Bessel}, \eqref{def_HsO_Bessel_norm_v1}, and \eqref{def_HsO_Bessel_inner_v1}, leading to the same Sobolev space $\HsRN$ and equivalent norms. The following result adapted from \cite{McLean_2000} relates the Bessel potential spaces $\HsRN$ to the spaces $\HmO$ defined above.

\begin{proposition}
For all $s =m \in \N$ the Sobolev spaces $\HsRN$ as defined in \eqref{def_HmO_Da_full} and \eqref{def_HsO_Bessel} are equal, and the corresponding norms defined in \eqref{def_HmO_Da_full_norm} and \eqref{def_HsO_Bessel_norm_v1} are equivalent.
\end{proposition}    
\begin{proof}
The proof is the same as in \cite[Theorem~3.16]{McLean_2000}, with a minor modification dealing with the additional factor $4 \pi^2$ in the definition of the inner product \eqref{def_HsO_Bessel_inner_v1}.
\end{proof}   

Instead of \eqref{def_HsO_Bessel_norm_v1} the spaces $\HsRN$ are often equipped with an equivalent norm:

\begin{proposition}
For all $s \geq 1$ an equivalent norm to $\norm{\cdot}_\HsRN$ is given by
    \begin{equation}\label{def_HsO_Bessel_norm_v2}
        \newnorm{u}_\HsRN 
        := 
        \norm{\kl{1+\kl{2\pi\abs{\cdot}}^{2s}}^\frac{1}{2} \F u}_\LtRN 
        =
        \kl{\int_\RN \kl{1+\kl{2\pi\abs{\xi}}^{2s}} \abs{(\F u)(\xi)}^2 \, d\xi }^{1/2} \,.
    \end{equation}
\end{proposition}
\begin{proof}
Let $s\geq 1$ be fixed. In order to show the equivalence between the norms defined in \eqref{def_HsO_Bessel_norm_v1} and \eqref{def_HsO_Bessel_norm_v2} it suffices to show that there exist constants $C_1,C_2 > 0$ such that
    \begin{equation}\label{ineq_Bessel_mult}
        C_1 \kl{1+\kl{2\pi\abs{\xi}}^{2s}} 
        \leq  
        \kl{1+4\pi^2\abs{\xi}^{2}}^s
        \leq
        C_2 \kl{1+\kl{2\pi\abs{\xi}}^{2s}} \,,
        \qquad
        \forall \, \xi \in \RN \,,
    \end{equation}
since then it follows that
    \begin{equation*}
        C_1 \newnorm{u}_\HsRN^2 
        \leq 
        \norm{u}_\HsRN^2 
        \leq 
        C_2 \newnorm{u}_\HsRN^2 \,,
        \qquad
        \forall \, u \in \HsRN \,.
    \end{equation*}
We start with the second inequality in \eqref{ineq_Bessel_mult}: since $x \mapsto x^s$ is convex on $\R_0^+$, we have
    \begin{equation*}
        \kl{1+4\pi^2\abs{\xi}^2}^s 
        =
        2^s \kl{\frac{1}{2}+\frac{1}{2}4\pi^2\abs{\xi}^2}^s
        \leq
        2^{s-1} \kl{1 +\kl{2\pi\abs{\xi}}^{2s}} \,,
        \qquad
        \forall \, \xi \in \RN \,,
    \end{equation*}
and thus the second inequality in \eqref{ineq_Bessel_mult} holds with $C_2 = 2^{s-1}$. On the other hand, since
    \begin{equation*}
        1 \leq \kl{1+4\pi^2\abs{\xi}^2}^s \,, 
        \qquad
        \text{and}
        \qquad
        \kl{2\pi\abs{\xi}}^{2s} \leq \kl{1+4\pi^2\abs{\xi}^2}^s  \,,
        \qquad
        \forall \, \xi \in \RN \,,
    \end{equation*}
it follows that
    \begin{equation*}
        1 + \kl{2\pi\abs{\xi}}^{2s} \leq 2\kl{1+4\pi^2\abs{\xi}^2}^s \,,
        \qquad
        \forall \, \xi \in \RN \,.
    \end{equation*}
Hence, the first inequality in \eqref{ineq_Bessel_mult} holds with $C_1 = 1/2$, which concludes the proof.
\end{proof}

Note that the norm $\newnorm{\cdot}_\HsRN$ defined in \eqref{def_HsO_Bessel_norm_v2} is induced by the inner product
    \begin{equation}\label{def_HsO_Bessel_inner_v2}
        \spr{u,v}_\HsRN := \int_\RN \kl{1+\kl{2\pi\abs{\xi}}^{2s}} (\F u)(\xi)  \overline{(\F v)(\xi)} \, d\xi \,.
    \end{equation}

% Subsection - Adjoint Sobolev embedding operators and Hilbert scales
\subsection{Adjoint Sobolev embedding operators and Hilbert scales}\label{subsect_Hilbert_scales}

We now return to the embedding operator $E_s$ defined in \eqref{embedding} for all $0 \leq s \in \R$, i.e.,
    \begin{equation*}
        E_s \, : \, \HsO \to \LtO \,,
        \qquad
        u \mapsto E_s u := u \,.
    \end{equation*}
Whenever we consider only Sobolev spaces $\HmO$ of integer order $m \in \N$, we write
    \begin{equation*}
        E_m \, : \, \HmO \to \LtO \,,
        \qquad
        u \mapsto E_m u := u \,,
    \end{equation*}
and similarly, when working with the Sobolev spaces $\HmzO$ with $m \in \N$, we define
    \begin{equation*}
        E_{m,0} \, : \, \HmzO \to \LtO \,,
        \qquad
        u \mapsto E_{m,0} u := u \,.
    \end{equation*}
For further reference, the proper definition of an adjoint operator is given in 
\begin{definition}\label{def_Es_adj}
Let $A : X \to Y$ be a bounded, linear operator between two Hilbert spaces $X$ and $Y$. Then the adjoint operator $A^* : Y \to X$ of $A$ is defined by the relation
	\begin{equation*}
		\spr{A^* u, v}_X = \spr{u, A v}_Y \,,
		\qquad
		\forall \, u \in Y \,, v \in X \,.
	\end{equation*}
Hence, for $E_s: \HsO \to \LtO$ the adjoint $E_s^* : \LtO \to \HsO$ is defined by 
	\begin{equation*}
		\spr{E_s^*u,v}_\HsO = \spr{u,E_s v}_\LtO = \spr{u, v}_\LtO \,,
		\qquad
		\forall \, u \in \LtO \,, v \in \HsO \,,
\end{equation*}
and analogously for the embeddings $E_{m} : \HmO \to \LtO$ and $E_{m,0} : \HmzO \to \LtO$.
\end{definition}

It follows from its definition that $E_s^*$ is a bounded linear operator, and it may even be compact, depending on properties of the domain $\Omega$; see e.g.~\cite{Evans_1998,McLean_2000,Adams_1970,Necas_2011}. For example, if $\Omega$ is a bounded domain with a Lipschitz continuous boundary, then it follows from \cite[Theorem~1.4]{Necas_2011} that $E_m : \HmO \to \LtO$ and thus also $E_m^*$ is compact for all $m \in \N$. Since $\HmzO \subset \HmO$ the same is true for $E_{m,0} : \HmzO \to \LtO$.

\begin{remark}
Note that the definition of the operator $E_s^*$ given in Definition~\ref{def_Es_adj} implicitly depends on the specific choice of the inner products on $\HsO$ and $\LtO$. This is important to keep in mind when working with one of the equivalent norms on $\HsO$ discussed above, since they are induced by different inner products and thus the corresponding operator $E_s^*$ is generally different. Hence, when using the notation $\norm{\cdot}_\HsO$ and $\spr{\cdot,\cdot}_\HsO$ to denote a norm and inner product on $\HsO$, respectively, we always have to specify which one of the different equivalent norms and corresponding inner products is meant. Whenever we do not explicitly specify the norm and corresponding inner product, the considered result holds independently of the specific choice.
\end{remark}

Next, we connect the embedding operator $E_s$ to the theory of Hilbert scales, revealing another link between its adjoint $E_s^*$ and the theory of inverse problems. For this, we restate some results found e.g.\ in \cite{Egger_Neubauer_2005,Neubauer_1988,Engl_Hanke_Neubauer_1996,Krein_Petunin_1966}. For this, let $L : D(L) \subseteq X \to X$ be a densely defined, selfadjoint, strictly positive operator on a Hilbert space $X$ such that 
    \begin{equation*}
        \norm{Lu}_X \geq \norm{u}_X \,,
        \qquad
        \forall \, u \in D(L) \,.
    \end{equation*}
Furthermore, for all $s \geq 0$ and $u , v \in \bigcap_{k=0}^\infty D(L^k)$ we define the inner product
    \begin{equation}\label{def_inner_Xs}
        \spr{u,v}_\Xs := \spr{L^s u,L^s v}_X \,.
    \end{equation}
Now, for any $s \geq 0$ the space $\Xs$ is defined as the completion of $\bigcap_{k=0}^\infty D(L^k)$ with respect to the norm $\norm{\cdot}_\Xs$ induced by the inner product $\spr{\cdot,\cdot}_\Xs$ defined in \eqref{def_inner_Xs}. Furthermore, for any $s < 0$ the space $X_s$ is defined as the dual space of $X_{-s}$. The collection of these spaces, i.e., $(\Xs)_{s\in\R}$, is called a Hilbert scale induced by the operator $L$.

It was shown in \cite{Krein_Petunin_1966} that the Sobolev spaces $\HsRN$ form a Hilbert scale. However, when the domain $\Omega \subseteq \RN$ is bounded, this is no longer true due to boundary conditions \cite{Neubauer_1988}. The proof of this interesting fact relies on the following result \cite[Proposition~2.1]{Neubauer_1988}:

\begin{proposition}
Let $X_1$ and $X_2$ be real Hilbert spaces such that $X_2$ is dense in $X_1$ and $\norm{u}_{X_1} \leq \norm{u}_{X_2}$ for all $u \in X_2$. Then $L: D(L) (\subset X_1) \to X_1$ is a densely defined, selfadjoint, strictly positive operator with $D(L) = X_2$ satisfying 
    \begin{equation*}
        \norm{L u}_{X_1} \geq \norm{u}_{X_2} \,, 
        \qquad
        \spr{Lu,Lv}_{X_1} = \spr{u,v}_{X_2} \,,
        \qquad
        \forall u,v \in X_2 \,,
    \end{equation*}
if and only if
    \begin{equation*}
        L = (I^*)^{-1/2} \,,
    \end{equation*}
where $I : X_2 \to X_1$ denotes the embedding operator and $I^* : X_1 \to X_2$ denotes its adjoint. Since $I^*$ is selfadjoint from $X_1 \to X_1$, the operator $(I^*)^{1/2}$ is well-defined.
\end{proposition}

From the above proposition we obtain the following 
\begin{corollary}
For any $s \geq 0$ let $E_s : \HsO \to \LtO$ and $E_s^*:\LtO \to \HsO$ be the embedding and its adjoint as defined in \eqref{embedding} and Definition~\ref{def_Es_adj}, respectively. Then
    \begin{equation*}
        \spr{u,v}_\HsO 
        = \spr{(E_s^*)^{-1/2}u,(E_s^*)^{-1/2}v}_\LtO 
        = \spr{(E_s^*)^{-1} u,v}_\LtO \,,
        \qquad 
        \forall \, u,v \in \HsO \,,
    \end{equation*}
as well as
    \begin{equation}\label{norm_Hs_Es}
        \norm{u}_\HsO = \norm{(E_s^*)^{-1/2} u}_\LtO \,,
        \qquad 
        \forall \, u \in \HsO \,.
    \end{equation}
\end{corollary}

% % % % % % % % % % % % % % % % % % % % %
% Section - Adjoint Embeddings and BVPs %
% % % % % % % % % % % % % % % % % % % % %
\section{Adjoint Embeddings and BVPs}\label{sect_PDEs}

In this section, we consider some representations of the adjoint embedding $E_s^*$ in terms of the solution of variational problems related to weak solutions of certain boundary value problems (BVPs). Our first result is a direct consequence of Definition~\ref{def_Es_adj}.

\begin{proposition}
Let $0\leq s \in \R$, $u \in \LtO$, and define the bilinear and linear forms
    \begin{equation}\label{def_a_l}
        a(z,v) := \spr{z,v}_\HsO \,,
        \qquad
        \text{and}
        \qquad
        l_u(v) := \spr{u,v}_\LtO \,,
        \qquad
        \forall \, v, z \in \HsO \,.
    \end{equation}
Then the element $E_s^* u$ is given as the unique solution $z \in \HsO$ of
    \begin{equation}\label{var_prob}
        a(z,v) = l_u(v) \,, 
        \qquad \forall \, v \in \HsO \,.
    \end{equation}
Similarly, for $m \in \N$ the element $E_{m,0}^* u$ is the unique solution $z \in \HmzO$ of
    \begin{equation}\label{var_prob_hom}
        a(z,v) = l_u(v) \,, 
        \qquad \forall \, v \in \HmzO  \,.
    \end{equation} 
\end{proposition}
\begin{proof}
This follows from the Definition~\ref{def_Es_adj} of $E_s^*$ and $E_{m,0}^*$ and the definitions of the bilinear and linear forms. The uniqueness of the solutions of \eqref{var_prob} and \eqref{var_prob_hom} also follows from the Lax-Milgram Lemma \cite{Evans_1998}, the assumptions of which are trivially satisfied.
\end{proof}

Note that the above result holds for all equivalent norms and corresponding inner products on $\HsO$ and $\HmzO$, as long as they are used consistently both in \eqref{def_a_l} and Definition~\ref{def_Es_adj}. Hence, in practice the adjoint embedding operators $E_s^*$ and $E_{m,0}^*$ are often implemented using suitable finite element discretizations of the variational problems \eqref{var_prob} and \eqref{var_prob_hom}; cf.~Section~\ref{sect_discrete}. This is also motivated by the fact that for $m \in \N$ and for particular choices of norms on $\HmO$ and $\HmzO$, the variational problems \eqref{var_prob} and \eqref{var_prob_hom} are related to weak solutions of certain BVPs. For making this more precise, we first consider the following proposition adapted from \cite[Theorem~10.2]{Agmon_1965}.

\begin{proposition}\label{prop_Agmon}
Let $m \in \N$ and let $\Omega \subset \RN$ be a bounded domain with a sufficiently smooth boundary $\partial \Omega$. Furthermore, let the constants $c_\alpha \in \Kl{0,1}$ be such that for all $\abs{\alpha} = m$ there holds $c_\alpha = 1$, and let the linear differential operator $B$ be defined by
    \begin{equation*}
        B z := \sum\limits_{0\leq \abs{\alpha} \leq m} (-1)^{\abs{\alpha}} c_{\alpha} D^{2\alpha} z \,.
    \end{equation*}
Then there exist $m$ linear differential operators $N_k$ of order $k$, $m \leq k \leq 2m-1$, i.e.,  
    \begin{equation*}
        N_k z = \sum\limits_{0\leq \abs{\alpha} \leq k} d_{k,\alpha} D^{\alpha} z \,,
    \end{equation*}
with functions $d_{k,\alpha}$ defined on $\partial \Omega$, such that for all $v,z \in C^{2m}(\overline{\Omega})$ there holds
    \begin{equation}\label{eq_var_diff}
        \sum\limits_{0\leq \abs{\alpha} \leq m} c_{\alpha} \spr{D^\alpha z, D^\alpha v }_\LtO 
        = 
        \spr{Bz,v}_\LtO 
        +
        \sum\limits_{j=0}^{m-1} \int_{\partial \Omega} \overline{\frac{\partial^j v}{\partial \nv^j}} N_{2m-1-j} z \, dS \,. 
    \end{equation}
Here, $\partial^j / \partial \nv^j$ denotes the $j$th derivative in the direction of the exterior normal vector $\nv$ of $\partial \Omega$, and the surface $\partial \Omega$ is non-characteristic for the differential operators $N_k$. 
\end{proposition}

For $m \in \N$ and coefficients $c_\alpha \in \Kl{0,1}$ with $c_\alpha = 1$ for $\abs{\alpha} = m$ consider the BVP
    \begin{equation*}
    \begin{split} 
        \sum\limits_{0\leq \abs{\alpha} \leq m} (-1)^{\abs{\alpha}} c_{\alpha} D^{2\alpha} z(x) 
        &= 
        u(x) \,, 
        \qquad \forall \, x \in \Omega\,,
        \\
        N_{2m-1-j} z(x) &= 0 \,, 
        \qquad \forall \, x \in \partial \Omega \,, \, 0 \leq j \leq m-1 \,.
    \end{split}    
    \end{equation*}
Due to \eqref{eq_var_diff} the variational problem associated to this boundary value problem is  
    \begin{equation*}
        \sum\limits_{0\leq \abs{\alpha} \leq m} c_{\alpha} \spr{D^\alpha z, D^\alpha v}_\LtO
        = 
        \spr{u,v}_\LtO \,,
        \qquad
        \forall \,
        v \in \HmO \,.
    \end{equation*}
Comparing this with \eqref{var_prob} and the definition \eqref{def_HmO_Da_full_inner} of $\spr{\cdot,\cdot}_\HmO$ we obtain the following

\begin{proposition}\label{prop_PDE_HmO_Da_full}
Let $m \in \N$ and let $\Omega \subset \RN$ be a bounded domain with a sufficiently smooth boundary $\partial \Omega$. Furthermore, let $\HmO$ be equipped with the norm $\norm{\cdot}_\HmO$ defined in \eqref{def_HmO_Da_full_norm} and let $\spr{\cdot,\cdot}_\HmO$ denote its corresponding inner product given in \eqref{def_HmO_Da_full_inner}. Then there exist $m$ linear differential operators $N_k$ of order $k$ with $m \leq k \leq 2m-1$ which are non-characteristic on $\partial \Omega$, such that for each $u \in \LtO$ the element $E_m^* u$ is the unique solution $z$ of the variational problem associated to the BVP
    \begin{equation}\label{PDE_norm_full}
    \begin{split}
         \sum\limits_{0\leq \abs{\alpha} \leq m} (-1)^{\abs{\alpha}} D^{2\alpha} z(x) &= u(x) \,,
         \qquad
         \forall \, x \in \Omega \,,
         \\
         N_{2m-1-j} z(x) &= 0 \,, 
        \qquad \forall \, x \in \partial \Omega \,, \, 0 \leq j \leq m-1 \,.
    \end{split}
    \end{equation}
\end{proposition}
\begin{proof}
Applying Proposition~\ref{prop_Agmon} with $c_\alpha = 1$ for all $\alpha$ we find that there exist $m$ linear differential operators $N_k$ of order $k$ with $m\leq k \leq 2m-1$ such that there holds
    \begin{equation}\label{eq_helper_4}
        \sum\limits_{0\leq \abs{\alpha} \leq m} \spr{D^\alpha z, D^\alpha v }_\LtO 
        = 
        \spr{Bz,v}_\LtO 
        +
        \sum\limits_{j=0}^{m-1} \int_{\partial \Omega} \overline{\frac{\partial^j v}{\partial \nv^j}} N_{2m-1-j} z \, dS \,, 
    \end{equation}
for all $v,z \in C^{2m}(\overline{\Omega})$, where the linear differential operator $B$ is defined by
    \begin{equation*}
        B z := \sum\limits_{0\leq \abs{\alpha} \leq m} (-1)^{\abs{\alpha}} D^{2\alpha} z \,.
    \end{equation*}
Next, consider the variational problem associated to the BVP \eqref{PDE_norm_full}, which is obtained by multiplying the differential equation with a test function $v$ and integrating, i.e.,
    \begin{equation*}
        \spr{Bz,v}_\LtO = \spr{u,v}_\LtO \,.
    \end{equation*}
Together with \eqref{eq_helper_4} and using the boundary conditions $N_{2m-1-j} z(x) = 0$ for all $x \in \partial \Omega$ and $0\leq j \leq m-1$ we find that the variational problem associated to \eqref{PDE_norm_full} is given by
    \begin{equation*}
        \sum\limits_{0\leq \abs{\alpha} \leq m} \spr{D^\alpha z, D^\alpha v }_\LtO =
        \spr{u,v}_\LtO \,,
        \qquad
        \forall \, v \in \HmO \,.
    \end{equation*}
Comparing this with \eqref{var_prob} and the definition \eqref{def_HmO_Da_full_inner} of $\spr{\cdot,\cdot}_\HmO$ find that $E_m^*u$ is in fact the unique solution of exactly this variational problem, which yields the assertion.
\end{proof}

We demonstrate how to derive explicit expressions for the operators $N_{2m-1-j}$ in

\begin{example}\label{example_PDE_01}
Let the domain $\Omega \subset \RN$ be bounded with $\partial \Omega \in C^1$, let $\HmO$ be equipped with the norm \eqref{def_HmO_Da_full_norm} corresponding to the inner product \eqref{def_HmO_Da_full_inner}, and let $E_1 : H^1(\Omega) \to \LtO$ denote the embedding operator. Then Green's formula \cite{Evans_1998} yields
    \begin{equation*}
        \spr{\sum\limits_{\abs{\alpha} =1} (-1)^{\abs{\alpha}} D^{2\alpha} z , v }_\LtO  
        =
        \spr{-\laplace z , v}_\LtO 
        =
        \sum\limits_{\abs{\alpha} = 1} \spr{D^\alpha z, D^\alpha v}_\LtO 
        -
        \int_{\partial\Omega} v \frac{\partial z}{\partial \nv} \, dS \,,
    \end{equation*}
and it follows by comparing this with \eqref{eq_var_diff} that $N_1 z = \partial z /\partial \nv$. Due to Proposition~\ref{prop_PDE_HmO_Da_full} for each $u \in \LtO$ the element $E_1^* u$ is thus given as the unique weak solution $z$ of  
    \begin{equation}\label{PDE_example_Green}
    \begin{split}
        - \laplace z(x) + z(x) &= u(x) \,, 
        \qquad 
        \forall \, x \in \Omega \,,
        \\
        \frac{\partial}{\partial \nv} z(x) &=0 \,,
        \qquad
        \forall \, x \in \partial \Omega \,.
    \end{split}
    \end{equation}
If the boundary regularity $\partial \Omega \in C^2$ holds, then it follows that $E_1^* u = z \in H^2(\Omega)$, and thus $E_1^* u$ is also a solution of \eqref{PDE_example_Green}, see e.g.\ \cite[Chapter~6.3, Theorem~4]{Evans_1998}. This applies in particular to $\Omega = (a,b) \subset \R$. If the boundary $\partial\Omega$ is not $C^2$, then it is still often possible to derive regularity estimates for the weak solution of \eqref{PDE_example_Green}, which typically lead to estimates of the form $E_1^*u = z \in H^s(\Omega)$ for some $s \in [1,2]$; see e.g.\ \cite{Savare_1997,Costable_Dauge_Nicaise_2010,Gilbarg_Trudinger_1998,Necas_2011,McLean_2000}.
\end{example}

When using the equivalent norm \eqref{def_HmO_Da_simple_norm} on $\HmO$ we obtain the following

\begin{proposition}
Let $m \in \N$, let the domain $\Omega \subseteq \RN$ be bounded and satisfy the weak cone condition, and let the boundary $\partial \Omega$ be sufficiently smooth. Furthermore, let $\HmO$ be equipped with the norm $\newnorm{\cdot}_\HmO$ defined in \eqref{def_HmO_Da_simple_norm} and let $\spr{\cdot,\cdot}_\HmO$ denote its corresponding inner product given in \eqref{def_HmO_Da_simple_inner}. Then there exist $m$ linear differential operators $N_k$ of order $k$ with $m \leq k \leq 2m-1$ which are non-characteristic on $\partial \Omega$, such that for each $u \in \LtO$ the element $E_m^* u$ is the unique solution $z$ of the variational problem associated to the BVP
    \begin{equation}\label{PDE_norm_simple}
    \begin{split}
        (-1)^{m} \sum\limits_{\abs{\alpha} = m}  D^{2\alpha} z(x) + z(x) &= u(x) \,,
        \qquad
        \forall \, x \in \Omega \,,
        \\
        N_{2m-1-j} z(x) &= 0 \,, 
        \qquad \forall \, x \in \partial \Omega \,, \, 0 \leq j \leq m-1 \,.
    \end{split}
    \end{equation}
\end{proposition}
\begin{proof}
This follows analogously to Proposition~\ref{prop_PDE_HmO_Da_full}, now using the choice $c_\alpha = 0$ for all multiindices $0 < \abs{\alpha} < m$, as well as $c_\alpha = 1$ for all $\abs{\alpha} = m$ and $\abs{\alpha} = 0$.
\end{proof}

The commonly encountered special case $N=1$ is considered in the following

\begin{example}
Let $m \in \N$, let the domain $\Omega \subset \R$ be bounded, let $\HmO$ be equipped with the norm \eqref{def_HmO_Da_simple_norm}, and let $E_m : \HmO \to \LtO$ be the embedding operator. Then
    \begin{equation*}
        \spr{ (-1)^m D^{2m} z, v }_\LtO 
        =
        \spr{D^m z, D^m v}_\LtO
        +
        \sum_{j=0}^{m-1} (-1)^{m+j}\int_{\partial \Omega} D^{2m-1-j} z D^{j} v \, \nv \, dS \,,
    \end{equation*}
and it follows by comparing this with \eqref{eq_var_diff} that $N_{k} = (-1)^{k+2-m} D^{k}$. Hence, due to Proposition~\ref{prop_PDE_HmO_Da_full} for each $u \in \LtO$ the element $E_m^* u$ is the unique weak solution $z$ of 
    \begin{equation*}
    \begin{split}
        (-1)^m D^{2m} z (x) + z(x) &= u(x) \,, 
        \qquad 
        \forall \, x \in \Omega \,,
        \\
        D^j z (x) &=0 \,,
        \qquad
        \forall \, x \in \partial \Omega \,, \, 0 \leq j \leq m-1 \,.
    \end{split}
    \end{equation*}
Using partial integration it can be seen that $E_m^* u = z \in H^{2m}(\Omega)$; see also \cite{Gilbarg_Trudinger_1998,Necas_2011,Agmon_1965}.
\end{example}

Furthermore, when using \eqref{def_seminorm_Hm} as an equivalent norm on $\HmzO$ we obtain

\begin{proposition}\label{prop_PDE_HmO_seminorm}
Let $m \in \N$ and let the domain $\Omega \subseteq \RN$ be bounded with a sufficiently smooth boundary $\partial \Omega$. Furthermore, let $\HmzO$ be equipped with the equivalent seminorm $\abs{\cdot}_\HmO$ defined in \eqref{def_seminorm_Hm} and let $\spr{\cdot,\cdot}_\HmO$ denote its corresponding inner product given in \eqref{def_seminorm_Hm_inner}. Then for each $u \in \LtO$ the element $E_{m,0}^* u$ is the unique solution $z \in \HmzO$ of the variational problem associated to the generalized Dirichlet BVP
    \begin{equation}\label{PDE_seminorm}
    \begin{split}
         (-1)^{m} \sum\limits_{\abs{\alpha} = m}  D^{2\alpha} z(x) &= u(x) \,,
         \qquad
         \forall \, x \in \Omega \,,
         \\
         D^\alpha z(x) &= 0 \,,
        \qquad
        \forall \, x \in \partial \Omega \,, 0 \leq \abs{\alpha} \leq m-1 \,.
    \end{split}
    \end{equation}
Furthermore, given the boundary regularity $\partial \Omega \in C^{2m}$ there holds $E_{m,0}^*u = z \in H^{2m}(\Omega)$.
\end{proposition}
\begin{proof}
The proof is analogous to the proof of Proposition~\ref{prop_PDE_HmO_Da_full}, with the difference that now the boundary integrals vanish due to the choice of $\HmzO$ for the test space. The regularity $z \in H^{2m}(\Omega)$ follows from \cite[Theorem~9.8]{Agmon_1965}.
\end{proof}

Due to the BVP characterizations \eqref{PDE_norm_full}, \eqref{PDE_norm_simple}, and \eqref{PDE_seminorm}, it follows that the adjoint embedding operator $E_m^*$ can be efficiently computed numerically via the different available methods for solving elliptic BVPs; see e.g.\ \cite{Jung_Langer_2012,Gilbarg_Trudinger_1998,Necas_2011,McLean_2000,Steinbach_2007}. Furthermore, the study of the properties of solutions of these BVPs directly translates to properties of the adjoint embedding operator $E_m^*$. As we have seen, these include in particular regularity estimates in dependence on the domain $\Omega$, which indicate that for $u \in \LtO$ the element $E_m^* u$ typically belongs to a Sobolev space with a higher order than $m$.

% % % % % % % % % % % % % % % % % % % % % % % % % % % %
% Section - Characterizations via Fourier Transforms  %
% % % % % % % % % % % % % % % % % % % % % % % % % % % %
\section{Characterizations via Fourier Transforms}\label{sect_fourier_transform}

In this section, we consider the Bessel potential spaces $\HsRN$ for arbitrary $0\leq s \in \R$ and present some representations of the adjoint embedding operator $E_s^*$ in terms of the Fourier transform. We start with the following result generalized from \cite[Lemma~3.1]{Ramlau_Teschke_2004_1}.

\begin{proposition}\label{prop_Bessel_v1}
Let $0 \leq s \in \R$, let $\HsRN$ be equipped with the norm $\norm{\cdot}_\HsRN$ defined in \eqref{def_HsO_Bessel_norm_v1} and let $\spr{\cdot,\cdot}_\HsRN$ denote its corresponding inner product given in \eqref{def_HsO_Bessel_inner_v1}. Then for each $u \in \LtRN$ the element $E_s^* u$ is given by
    \begin{equation}\label{eq_Es_Fourier_v1}
         (E_s^*u)(x) = \FI \kl{ \kl{1+4\pi^2\abs{\cdot}^2}^{-s} \F u }(x) \,.
    \end{equation}
\end{proposition}
\begin{proof}
Let $0\leq s \in \R$ and let $u \in \LtRN$ and $v \in \HsRN$ be arbitrary but fixed. Then due to the definition \eqref{def_HsO_Bessel_inner_v1} of the inner product $\spr{\cdot,\cdot}_\HsRN$ there holds
    \begin{equation*}
    \begin{split}
        \spr{u,v}_\LtRN 
        &=
        \int_\RN (\F u)(\xi) \overline{(\F v)(\xi)} \, d\xi
        \\
        & = 
        \int_\RN \kl{1+4\pi^2\abs{\xi}^2}^s \kl{ \kl{1+4\pi^2\abs{\xi}^2}^{-s} (\F u)(\xi)}  \overline{(\F v)(\xi)} \, d\xi \,,
    \end{split}
    \end{equation*}
and thus again due to \eqref{def_HsO_Bessel_inner_v1} that
    \begin{equation*}
        \spr{u,v}_\LtRN 
        = 
        \spr{ \FI\kl{ \kl{1+4\pi^2\abs{\cdot}^2}^{-s} \F u}, v }_\HsRN \,.
    \end{equation*}
Comparing this with Definition~\ref{def_Es_adj} we find that
    \begin{equation*}
         (E_s^*u)(x) = \FI \kl{ \kl{1+4\pi^2\abs{\cdot}^2}^{-s} \F u }(x) \,, 
    \end{equation*}
which yields the assertion.
\end{proof}

When using the equivalent norm $\newnorm{\cdot}_\HsRN$ on $\HsRN$ we obtain the following

\begin{proposition}
Let $1 \leq s \in \R$, let $\HsRN$ be equipped with the norm $\newnorm{\cdot}_\HsRN$ defined in \eqref{def_HsO_Bessel_norm_v2} and let $\spr{\cdot,\cdot}_\HsRN$ denote its corresponding inner product given in \eqref{def_HsO_Bessel_inner_v2}. Then for each $u \in \LtRN$ the element $E_s^* u$ is given by
    \begin{equation*}
         (E_s^*u)(x) = \FI \kl{ \kl{1+\kl{2\pi\abs{\cdot}}^{2s}}^{-1} \F u }(x) \,. 
    \end{equation*}
\end{proposition}
\begin{proof}
The proof is completely analogous to the proof of Proposition~\ref{prop_Bessel_v1}
\end{proof}

In this context, we mention the Bessel potential operator of order $s$, defined by \cite{Aronszajn_Smith_1961}:

\begin{definition}
For all $s \in \R$ the Bessel potential operator $(I-\laplace)^{-s/2}$ is defined by
    \begin{equation}\label{def_Bessel_potential}
        (I-\laplace)^{-s/2}u(x) := \FI\kl{\kl{1+4\pi^2\abs{\cdot}^2}^{-s/2} \F u}(x) \,.
    \end{equation}
\end{definition}
It follows from the definition \eqref{def_HsO_Bessel} of the Sobolev spaces $\HsRN$ that there holds
    \begin{equation}\label{Bessel_potential_HsRN}
        u \in \HsRN 
        \qquad
        \Longleftrightarrow
        \qquad
        (I-\laplace)^{s/2} u \in \LtRN \,,
    \end{equation}
which explains their alternative name of Bessel potential spaces. Furthermore, we have the following expression of the adjoint embedding $E_s^*$ in terms of the Bessel potential:

\begin{proposition}
Let $0 \leq s \in \R$, let $\HsRN$ be equipped with the norm $\norm{\cdot}_\HsRN$ defined in \eqref{def_HsO_Bessel_norm_v1} and let $\spr{\cdot,\cdot}_\HsRN$ denote its corresponding inner product given in \eqref{def_HsO_Bessel_inner_v1}. Then for each $u \in \LtRN$ the element $E_s^* u$ is given by
    \begin{equation}\label{eq_Es_Bessel}
        E_s^* u = (I-\laplace)^{-s} u \,,
    \end{equation}
where $(I-\laplace)^{-s}$ denotes the Bessel potential operator defined in \eqref{def_Bessel_potential}.
\end{proposition}
\begin{proof}
This follows by comparing the characterization \eqref{eq_Es_Fourier_v1} of the adjoint embedding operator $E_s^*$ with the definition \eqref{def_Bessel_potential} of the Bessel potential operator $(I-\laplace)^{-s}$.
\end{proof}

The representation \eqref{eq_Es_Bessel} of the adjoint embedding operator $E_s^*$ should be compared to its BVP characterization given in Proposition~\ref{prop_PDE_HmO_Da_full}. The representation \eqref{eq_Es_Bessel} is also useful when we consider characterizations via spatial filters in the next section.

% % % % % % % % % % % % % % % % % % % % % % % % % %
% Section - Characterizations via Spatial Filters %
% % % % % % % % % % % % % % % % % % % % % % % % % %
\section{Characterizations via Spatial Filters}\label{sect_filters}

In the last sections we considered characterizations of $E_s^* u$ as the solution of certain boundary value problems with right hand side $u$, or via a dampening of the Fourier coefficients of $u$ by suitable factors. These characterizations indicate that the application of $E_s^*$ results in a smoothing or ``smearing out'' effect. We now make this observation more precise by characterizing $E_s^*$ in terms of spatial filters with suitable kernels. These results are closely related to Bessel potentials, which are considered in detail e.g.\ in \cite{Aronszajn_Smith_1961}.

\begin{proposition}
Let $0 < s \in \R$, let $\HsRN$ be equipped with the norm $\norm{\cdot}_\HsRN$ defined in \eqref{def_HsO_Bessel_norm_v1} and let $\spr{\cdot,\cdot}_\HsRN$ denote its corresponding inner product given in \eqref{def_HsO_Bessel_inner_v1}. Then for each $u \in \LtRN$ the element $E_s^* u$ is given by the convolution
    \begin{equation}\label{eq_Es_Gs}
        (E_s^*u)(x) = (G_{2s} \ast u)(x) \,,
    \end{equation}
where the convolution kernel $G_s : \RN \to \R$ is defined by
    \begin{equation}\label{def_Gs}
        G_s(x) := \FI \kl{ \kl{1+4\pi^2 \abs{\cdot}^2}^{-s/2} }(x) \,. 
    \end{equation}
\end{proposition}
\begin{proof}
Let $0 \leq s \in \R$ and $u \in \LtRN$ be arbitrary but fixed. Due to \eqref{eq_Es_Fourier_v1} there holds
    \begin{equation*}
         (E_s^*u)(x) = \FI \kl{ \kl{1+4\pi^2\abs{\cdot}^2}^{-s} \F u }(x) \,.
    \end{equation*}
Hence, using the convolution property \eqref{Fourier_convolution} we obtain
     \begin{equation*}
         (E_s^*u)(x) = \kl{ \FI \kl{ \kl{1+4\pi^2\abs{\cdot}^2}^{-s} } \ast u}(x) \,, 
    \end{equation*} 
and thus together with the definition \eqref{def_Gs} of $G_s$ we obtain
    \begin{equation*}
        (E_s^*u)(x) = (G_{2s} \ast u)(x) \,,
    \end{equation*}
which yields the assertion.
\end{proof}

The function $G_s$ can be computed explicitly, as the following result from \cite{Aronszajn_Smith_1961} shows:
\begin{proposition}
For all $0 < s \in \R$ the function $G_s$ defined in \eqref{def_Gs} satisfies
    \begin{equation}\label{def_Gs_Bessel}
        G_s(x) 
        =
        \frac{1}{2^{\frac{N+s-2}{2}} \pi^{\frac{N}{2}} \Gamma\kl{\frac{s}{2}}   }
        K_{\frac{N-s}{2}}(\abs{x}) \abs{x}^{\frac{s-N}{2}} \,,
    \end{equation}
where $\Gamma$ denotes the Gamma function and $K_{\frac{N-s}{2}}$ denotes the modified Bessel function of the third kind.
\end{proposition}

It was shown in \cite{Aronszajn_Smith_1961} that for all $s > 0$, the function $G_s$ given in \eqref{def_Gs_Bessel} is everywhere positive and decreasing in $\abs{x}$. Furthermore, $G_s$ is an analytic function of $\abs{x}$ except at $x = 0$. In addition, $G_s$ is integrable and thus its Fourier transform exists, thereby justifying \eqref{def_Gs} for all $s >0$. However, note that $G_s \in \LtRN$ if and only if $s > N/2$. In particular, note that \eqref{def_Gs} implies that $G_0$ formally equals the delta distribution, and thus \eqref{eq_Es_Gs} simplifies to $E_0^*u = u$ as expected. Moreover, $G_s$ is a Green's function of the Bessel potential operator. This can also be seen from \eqref{eq_Es_Bessel} and \eqref{eq_Es_Gs}, which yield
    \begin{equation*}
        (I-\laplace)^{-s} u = E_s^* u = G_{2s} \ast u \,.
    \end{equation*}
The asymptotic behaviour of the function $G_s$ is summarized from \cite{Aronszajn_Smith_1961} in the following

\begin{proposition}
Let the function $G_s$ be defined as in \eqref{def_Gs_Bessel}. 
    \begin{equation*}
    \text{Then for} \, \, x \to 0 \, :
    \qquad
    \begin{cases}
        G_s(x) \sim 
        \frac{\Gamma\kl{\frac{N-s}{2}}}{2^s \pi^{\frac{N}{2}} \Gamma\kl{\frac{s}{2}}} \abs{x}^{s-N}
        \,, 
        & s < N \,,
        \\
        G_s(x) \sim \frac{1}{2^{N-1} \pi^{\frac{N}{2}} \Gamma\kl{\frac{N}{2}}} \log{\frac{1}{\abs{x}}} 
        \,,
        & s = N \,,
        \\
        G_s(x) \sim \frac{\Gamma\kl{\frac{s-N}{2}}}{2^N \pi^{\frac{N}{2}} \Gamma\kl{\frac{s}{2}}} 
        \,, 
        & s > N \,, 
    \end{cases}    
    \end{equation*}
Furthermore, for all $s > 0$ and for $\abs{x} \to \infty$ there holds
    \begin{equation*}
        G_s(x) \sim \frac{1}{2^{\frac{N+s-1}{2}} \pi^{\frac{N-1}{2}} \Gamma(\frac{s}{2})} \abs{x}^{\frac{s-N-1}{2}} e^{-\abs{x}} \,.
    \end{equation*}
\end{proposition}

For integer values of $s$, simpler forms of $G_s$ not involving Bessel functions can often be found, e.g.\ with the help of computer algebra tools. Here, we provide the following

\begin{example}
For $N=1$ and $G_s$ defined as in \eqref{def_Gs_Bessel} it follows that
    \begin{equation*}
        G_2(x) = \frac{1}{2} e^{-\abs{x}} \,,
        \qquad
        \text{and}
        \qquad
        G_4(x) = \frac{1}{4} e^{-\abs{x}}(\abs{x} + 1) \,.
    \end{equation*}
Hence, if $E_1 : H^1(\R) \to L_2(\R)$ denotes the embedding operator, then there holds
    \begin{equation*}
        (E_1^* u)(x) = (G_2 \ast u)(x)
        =
        \frac{1}{2} \int_\R e^{-\abs{x-y}} u(y) \, dy \,,
        \qquad
        \forall \, u \in L_2(\R) \,,
    \end{equation*}
and analogously for $E_2 : H^2(\R) \to L_2(\R)$ it follows that
    \begin{equation*}
        (E_2^* u)(x) = (G_4 \ast u)(x)
        =
        \frac{1}{4} \int_\R e^{-\abs{x-y}} (\abs{x-y} + 1) u(y) \, dy \,,
        \qquad
        \forall \, u \in L_2(\R) \,.
    \end{equation*}
Thus, we see that the application of $E_s^*$ basically acts as a smoothing of the function.
\end{example}

% % % % % % % % % % % % % % % % % % % % % % % % % % % %
% Section - Characterizations via Wavelet Transforms  %
% % % % % % % % % % % % % % % % % % % % % % % % % % % %
\section{Characterizations via Wavelet Transforms}\label{sect_wavelets}

In this section, we consider some representations of the adjoint embedding operator $E_s^*$ in terms of wavelet transforms. For a detailed introduction and overview of the theory of wavelets we refer to the standard works \cite{Daubechies_1992,Meyer_1993}. Here, we only briefly review some relevant definitions and results, starting with the following definition adapted from \cite{Meyer_1993}. 

\begin{definition}
Let $\psi \in \LiR$ and let $\psi_{j,k}(x) := 2^{j/2} \psi(2^j x - k)$ for all $j,k\in\Z$. Then $\psi$ is called a (basic) wavelet of class $m \in \N_0$ if and only if the following conditions hold: 
    \begin{enumerate}
        \item For all $0 \leq \abs{\alpha} \leq m$ there holds $D^\alpha \psi \in \LiR$.
        \item For all $0 \leq \abs{\alpha} \leq m$ the derivatives $D^\alpha \psi$ decrease rapidly as $x \to \pm \infty$.
        \item For all $0 \leq n \leq m$ there holds $\int_{-\infty}^\infty x^n \psi(x) \, dx = 0$.  
        \item The set $\Kl{\psi_{j,k} }_{j,k\in\Z}$ forms an orthonormal basis of $\LtR$.
    \end{enumerate}
The function $\psi$ is also called the mother wavelet, and the $\psi_{j,k}$ are called wavelets.
\end{definition}

The above definition including regularity and orthonormality conditions is specific to wavelets of class $m$. In general, wavelets are defined as families of functions $\psi_{j,k}$ which are shifted and scaled versions of a mother wavelet $\psi$ and enjoy some type of basis property, but not necessarily orthonormality. This definition can be generalized to higher dimensions in different ways. E.g., for $N=2$ one could use the tensor products
    \begin{equation*}
        \psi_{j_1,j_2,k_1,k_2}(x,y) := \psi_{j_1,k_1}(x)\psi_{j_2,k_2}(y)
        =
        2^{(j_1+j_2)/2} \psi(2^{j_1} x - k_1)  \psi(2^{j_2} x - k_2)\,,
    \end{equation*}
and consider the wavelet family $\Kl{\psi_{j_1,j_2,k_1,k_2}}_{j_1,j_2,k_1,k_2 \in \Z}$, and similarly for all $N \geq 2$. Here, we use the more common generalization introduced below, where instead of multiple dilation factors $j_n$ only a single factor $j$, but multiple basic wavelets $\psi^\eps$ are used.

\begin{definition}
For $N \in \N$ and for all $j \in \Z$ let the sets $\Gamma_j$, $\Lambda_j$, and $\Lambda$ be defined by
    \begin{equation}\label{def_Gamma_Lambda}
        \Gamma_j := 2^{-j} \ZN \,,
        \qquad
        \Lambda_j := \Gamma_{j+1}\setminus \Gamma_j \,,
        \qquad
        \text{and}
        \qquad
        \Lambda := \kl{ \bigcup\limits_{j \in \Z} \Gamma_j  }\setminus \Kl{(0,\dots,0)} \,.
    \end{equation}
Furthermore, let the set $\M$ consisting of $(2^N-1)$ elements $\eps$ be defined by
    \begin{equation*}
        \M := \Kl{ \eps = \kl{\eps_1 \,, \dots \,, \eps_N} \in \Kl{0,1}^N \, \vert \, \eps \neq \kl{0,\dots,0}} \,, 
    \end{equation*}
and for each $\eps \in \M$ let $\psi^\eps$ be a basic wavelet. Then for each $\lambda \in \Lambda$ we define
    \begin{equation}\label{def_psi_lambda}
        \psi_\lambda (x) := 2^{Nj/2}\psi^\eps(2^j x- k) \,,
    \end{equation}
where  $j \in \Z$, $k \in \Z^N$, and $\eps \in \M$ are uniquely defined via $\lambda = 2^{-j} k + 2^{-j-1} \eps \in \Lambda$.
\end{definition}

The theory of wavelets is closely related to the concept of a multiresolution analysis: 

\begin{definition}
A family $\Kl{V_j}_{j \in \Z}$ of closed linear subspaces $V_j$ of $\LtRN$ is called a multiresolution approximation or multiresolution analysis of $\LtRN$, if and only if
    \begin{enumerate}
        \item For all $j \in \Z$ there holds $V_j \subset V_{j+1}$, as well as
            \begin{equation*}
                \bigcap_{j \in \Z} V_j = \Kl{0} \,,
                \qquad
                \text{and}
                \qquad
                \overline{ \bigcup_{j\in\Z} V_j } = \LtRN \,.
            \end{equation*}
        \item For all $u \in \LtRN$, all $j \in \Z$, and all $k \in \ZN$ there holds
            \begin{equation*}
                u(x) \in V_j 
                \quad
                \Longleftrightarrow 
                \quad
                u(2 x ) \in V_{j+1} \,,
                \qquad
                \text{and}
                \qquad
                u(x) \in V_0 
                \quad
                \Longleftrightarrow
                \quad
                u(x-k) \in V_0 \,.        
        \end{equation*}
        \item There exists $\phi \in V_0$ such that $\Kl{\phi(x-k)}_{k\in\ZN}$ is an orthonormal basis of $V_0$.
    \end{enumerate}
Furthermore, the multiresolution analysis $\Kl{V_j}_{j\in\Z}$ is called $r$-regular, if for each $m \in \N$ there exists a constant $C_m > 0$ such that for all $0 \leq \abs{\alpha} \leq r$ there holds
    \begin{equation*}
        \abs{D^\alpha \phi(x)} \leq C_m \kl{1+\abs{x}}^{-m} \,,
        \qquad
        \forall \, x \in \RN \,.
    \end{equation*}
\end{definition}

The function $\phi$ is also called scaling function or father wavelet. Note that instead of orthonormality, one could also assume $\Kl{\phi(x-k)}_{k\in\ZN}$ to be a Riesz basis. Since wavelet families are often linked to a multiresolution analysis, we now make the following

\begin{definition}
Let $\Kl{\psi_\lambda}_{\lambda \in \Lambda}$ be a wavelet family as defined in \eqref{def_psi_lambda} and let $\Kl{V_j}_{j \in \Z}$ be an $r$-regular multiresolution analysis for $\LtRN$ with scaling function $\phi$. Furthermore, for each $j \in \Z$ let $W_j$ denote the orthogonal complement of $V_j$ in $V_{j+1}$, such that $V_{j+1} = V_j \oplus W_j$. Moreover, for all $\lambda \in \Gamma_j$ let $\phi_\lambda(x) := 2^{Nj}\phi(2^j x- k)$ and assume that
    \begin{enumerate}
        \item The set $\Kl{\psi_\lambda}_{\lambda \in \Lambda_j}$ forms an orthonormal basis of $W_j$.
        \item The set $\Kl{\phi_\lambda}_{\lambda \in \Gamma_j}$ forms an orthonormal basis of $V_j$.
    \end{enumerate}
Then we say that $\Kl{\psi_\lambda}_{\lambda \in \Lambda}$ corresponds to the $r$-regular multiresolution analysis $\Kl{V_j}_{j\in\Z}$.
\end{definition}

It can be seen from the definition of the subspaces $W_j$ that there holds 
    \begin{equation*}
        \LtRN = \bigoplus\limits_{j\in\Z} W_j = V_0 \bigoplus\limits_{j=0}^\infty W_j \,.
    \end{equation*}
Hence, it follows that both the wavelet families $\Kl{\psi_\lambda}_{\lambda \in \Lambda}$ and $\Kl{\phi_\lambda}_{\lambda \in \Gamma_0} \cup \Kl{\psi_\lambda}_{j \geq 0 \,, \lambda \in \Gamma_j }$ form an orthonormal basis of $\LtRN$. Hence, for each function $u \in \LtRN$ there holds
    \begin{equation}\label{eq_rec_wavelet}
        u = \sum\limits_{\lambda \in \Lambda}
        \spr{u,\psi_\lambda}_\LtRN \psi_\lambda  
        = \sum\limits_{\lambda \in \Gamma_0} \spr{u,\phi_\lambda}_\LtRN \phi_\lambda + \sum\limits_{j\geq 0} \sum\limits_{\lambda \in \Lambda_j} \spr{u,\psi_\lambda}_\LtRN \psi_\lambda \,. 
    \end{equation}
These expansions can also be used to define equivalent norms on $\HsRN$, as the following theorem adapted from \cite{Meyer_1993, Daubechies_1992} shows.

\begin{theorem}
Let $0 \leq s \in \R$, let $\HsRN$ be equipped with the norm $\norm{\cdot}_\HsRN$ defined in \eqref{def_HsO_Bessel_norm_v1} and let $\spr{\cdot,\cdot}_\HsRN$ denote its corresponding inner product given in \eqref{def_HsO_Bessel_inner_v1}. Furthermore, let the orthonormal wavelet family $\Kl{\psi_\lambda}_{\lambda \in \Lambda}$ as defined in \eqref{def_psi_lambda} correspond to an $r$-regular multiresolution analysis of $\LtRN$ with $s < r$ and scaling function $\phi$. Then an equivalent norm for $\HsRN$ is given by
    \begin{align}
        \norm{u}_\HsRN 
        :=& \kl{\sum\limits_{j < 0}\sum\limits_{\lambda \in \Lambda_j} \abs{\spr{u,\psi_\lambda}_\LtRN}^2
        +
        \sum\limits_{j\geq 0} \sum\limits_{\lambda \in \Lambda_j} 2^{2js} \abs{\spr{u,\psi_\lambda}_\LtRN}^2 }^{1/2} 
        \label{def_HsO_wavelet_norm_v1} \\
        =&
        \kl{ \sum\limits_{\lambda \in \Gamma_0} \abs{\spr{u,\phi_\lambda}_\LtRN}^2 
        +
        \sum\limits_{j\geq 0} \sum\limits_{\lambda \in \Lambda_j} 2^{2js} \abs{\spr{u,\psi_\lambda}_\LtRN}^2}^{1/2} \,.
        \label{def_HsO_wavelet_norm_v2}
    \end{align}
\end{theorem}
\begin{proof}
As stated above, a proof of this theorem can e.g.\ be found in \cite{Meyer_1993}. Note that the equivalence of the expressions \eqref{def_HsO_wavelet_norm_v1} and \eqref{def_HsO_wavelet_norm_v2} follows from the fact that 
    \begin{equation*}
        \sum\limits_{j < 0}\sum\limits_{\lambda \in \Lambda_j} \abs{\spr{u,\psi_\lambda}}_\LtRN^2 
        =
        \sum\limits_{\lambda \in \Gamma_0} \abs{\spr{u,\phi_\lambda}_\LtRN}^2 \,,
    \end{equation*}
which is a consequence of the definition of the subspaces $V_j$ and $W_j$.
\end{proof}

Note that the norm $\norm{\cdot}_\HsRN$ defined in \eqref{def_HsO_wavelet_norm_v1} is induced by the inner product
    \begin{equation}\label{def_HsO_wavelet_inner_v1}
    \begin{split}
        \spr{u,v}_\HsRN &:= \sum\limits_{j < 0}\sum\limits_{\lambda \in \Lambda_j} \spr{u,\psi_\lambda}_\LtRN \overline{\spr{v,\psi_\lambda}}_\LtRN
        \\
        & \qquad +
        \sum\limits_{j\geq 0} \sum\limits_{\lambda \in \Lambda_j} 2^{2js} \spr{u,\psi_\lambda}_\LtRN
        \overline{\spr{v,\psi_\lambda}}_\LtRN \,,
    \end{split}
    \end{equation}
and that its equivalent version \eqref{def_HsO_wavelet_norm_v2} is induced by the inner product
    \begin{equation}\label{def_HsO_wavelet_inner_v2}
    \begin{split}
        \spr{u,v}_\HsRN &:= 
        \sum\limits_{\lambda \in \Gamma_0} \spr{u,\phi_\lambda}_\LtRN
        \overline{\spr{v,\phi_\lambda}}_\LtRN
        \\
        & \qquad + 
        \sum\limits_{j\geq 0} \sum\limits_{\lambda \in \Lambda_j} 2^{2js} \spr{u,\psi_\lambda}_\LtRN \overline{\spr{v,\psi_\lambda}}_\LtRN \,.
    \end{split}
    \end{equation}
When using the equivalent norm \eqref{def_HsO_wavelet_norm_v1} on $\HsRN$, we obtain the following representation of $E_s^*$, generalizing the one-dimensional case $N=1$ previously considered in \cite{Ramlau_Teschke_2004_1,Ramlau_2008}.

\begin{proposition}\label{prop_wavelet_v1}
Let $0 \leq s \in \R$, let $\HsRN$ be equipped with the norm $\norm{\cdot}_\HsRN$ defined in \eqref{def_HsO_wavelet_norm_v1} and let $\spr{\cdot,\cdot}_\HsRN$ denote its corresponding inner product given in \eqref{def_HsO_wavelet_inner_v1}. Furthermore, let the orthonormal wavelet family $\Kl{\psi_\lambda}_{\lambda \in \Lambda}$ defined in \eqref{def_psi_lambda} correspond to an $r$-regular multiresolution analysis of $\LtRN$ with $s < r$ and with a scaling function $\phi$. Then for each $u \in \LtRN$ the element $E_s^* u$ is given by
    \begin{equation}
        E_s^* u = \sum\limits_{j < 0}\sum\limits_{\lambda \in \Lambda_j} \spr{u,\psi_\lambda}_\LtRN \psi_\lambda
        +
        \sum\limits_{j\geq 0} \sum\limits_{\lambda \in \Lambda_j} 2^{-2js} \spr{u,\psi_\lambda}_\LtRN \psi_\lambda \,.
    \end{equation}
Moreover, if instead of \eqref{def_HsO_wavelet_norm_v1} the expression \eqref{def_HsO_wavelet_norm_v2} is used for the equivalent norm $\norm{\cdot}_\HsO$ on $\HsRN$ and if $\spr{\cdot,\cdot}_\HsRN$ denotes is corresponding inner product given in \eqref{def_HsO_wavelet_inner_v2}, then
    \begin{equation}
        E_s^* u = \sum\limits_{\lambda \in \Gamma_0 } \spr{u,\phi_\lambda}_\LtRN \phi_\lambda
        +
        \sum\limits_{j\geq 0} \sum\limits_{\lambda \in \Lambda_j} 2^{-2js} \spr{u,\psi_\lambda}_\LtRN \psi_\lambda \,.
    \end{equation}
\end{proposition}
\begin{proof}
Let $0\leq s \in \R$ and let $u \in \LtRN$ be arbitrary but fixed. Due to Definition~\ref{def_Es_adj}, we know that the adjoint embedding $E_s^*u$ is uniquely characterized by 
    \begin{equation*}
        \spr{E_s^*u,v}_\HsO = \spr{u, v}_\LtO \,,
        \qquad
        \forall v \in \HsO \,.
    \end{equation*}
Due to the definition \eqref{def_HsO_wavelet_inner_v1} of the inner product $\spr{\cdot,\cdot}_\HsRN$ this is equivalent to
    \begin{equation*}
    \begin{split}
        &\sum\limits_{j < 0}\sum\limits_{\lambda \in \Lambda_j} \spr{E_s^* u,\psi_\lambda}_\LtRN \spr{v,\psi_\lambda}_\LtRN
        +
        \sum\limits_{j\geq 0} \sum\limits_{\lambda \in \Lambda_j} 2^{2js}  \spr{E_s^* u,\psi_\lambda}_\LtRN
        \spr{v,\psi_\lambda}_\LtRN
        \\
        & \qquad =
        \sum\limits_{j < 0}\sum\limits_{\lambda \in \Lambda_j} \spr{u,\psi_\lambda}_\LtRN \spr{v,\psi_\lambda}_\LtRN
        +
        \sum\limits_{j\geq 0} \sum\limits_{\lambda \in \Lambda_j}  \spr{u,\psi_\lambda}_\LtRN 
        \spr{v,\psi_\lambda}_\LtRN \,,
    \end{split}    
    \end{equation*}
for all $v \in \HsO$. Due to the orthonormality of the functions $\psi_\lambda$, it thus follows that
    \begin{equation*}
    \begin{split}
        \spr{E_s^* u,\psi_\lambda}_\LtRN
        =
        \begin{cases}
            \spr{u,\psi_\lambda}_\LtRN \,,
            &
            j < 0 \,, \, \lambda \in \Lambda_j \,,
            \\
            2^{-2js} \spr{u,\psi_\lambda}_\LtRN \,,
            &
            j \geq 0 \,, \, \lambda \in \Lambda_j \,.
        \end{cases} 
    \end{split}
    \end{equation*}
Hence, together with the reconstruction formula \eqref{eq_rec_wavelet} we obtain
    \begin{equation*}
        E_s^* u = \sum\limits_{j < 0}\sum\limits_{\lambda \in \Lambda_j} \spr{u,\psi_\lambda}_\LtRN \psi_\lambda
        +
        \sum\limits_{j\geq 0} \sum\limits_{\lambda \in \Lambda_j} 2^{-2js} \spr{u,\psi_\lambda}_\LtRN \psi_\lambda \,,
    \end{equation*}
which yields the first part of the assertion. The second part follows analogously.
\end{proof}

We end this section with an example in the one-dimensional setting; cf.\ also \cite{Ramlau_Teschke_2004_1,Ramlau_2008}.

\begin{example}\label{example_wavelets_N1}
Consider the case $N=1$, let $\psi$ be a mother wavelet and let the wavelets $\psi_{j,k}(x) = 2^j \psi(2^j x - k)$ be such that $\Kl{\psi_{j,k}}_{j,k\in \Z}$ forms an orthonormal basis of $\LtR$. We now want to characterize the adjoint $E_s^*$ of the embedding $E_s : \HsR \to \LtR$ using these wavelets by invoking Proposition~\ref{prop_wavelet_v1}. For this, we first need to write the wavelet family $\Kl{\psi_{j,k}}_{j,k\in\Z}$ in the form $\Kl{\psi_\lambda}_{\lambda \in \Lambda}$ as defined in \eqref{def_psi_lambda}. This can be done via
    \begin{equation*}
        \psi_\lambda(x)  = \psi_{j,k}(x) = 2^{j/2} \psi(2^j x - k) \,.
    \end{equation*}
using the relation $\lambda = 2^{-j} k + 2^{-j-1} \in \Lambda$ with $\Lambda$ as defined in \eqref{def_Gamma_Lambda}. Furthermore, since $N=1$ it follows that $\M = \Kl{1}$, that $\Gamma_j = 2^{-j} \Z$, and thus $\Lambda_j = 2^{-j}(\Z+1/2)$. Hence,
    \begin{equation*}
        \lambda = 2^{-j} ( k + 2^{-1} ) \in \Lambda_j \quad
        \Longleftrightarrow 
        \quad 
        k \in \Z \,,
        \qquad \quad
        \forall \, j \in \Z \,.
    \end{equation*}
Next, assume that $\Kl{\psi_{j,k}}_{j,k\in\Z}$ corresponds to an $r$-regular multiresolution analysis of $\LtR$ with $r > s$. Furthermore, let $\phi$ denote the scaling function of the multiresolution analysis and let $\phi_{0,k}(x) := \phi(x-k)$. Then due to Proposition~\ref{prop_wavelet_v1}, if $\HsR$ is equipped with the norm \eqref{def_HsO_wavelet_norm_v1} and inner product \eqref{def_HsO_wavelet_inner_v1}, then for all $u \in \LtR$ there holds
    \begin{equation*}
        E_s^* u = \sum\limits_{j < 0}\sum\limits_{k\in\Z} \spr{u,\psi_{j,k}}_\LtR \psi_{j,k}
        +
        \sum\limits_{j\geq 0} \sum\limits_{k \in \Z} 2^{-2js} \spr{u,\psi_{j,k}}_\LtR \psi_{j,k} \,.
    \end{equation*}
Similarly, if $\HsR$ is equipped with the norm \eqref{def_HsO_wavelet_norm_v2} and inner product \eqref{def_HsO_wavelet_inner_v2} then
    \begin{equation*} 
        E_s^* u = \sum\limits_{k \in \Z} \spr{u,\phi_{0,k}}_\LtR \phi_{0,k}
        +
        \sum\limits_{j\geq 0} \sum\limits_{k \in \Z} 2^{-2js} \spr{u,\psi_{j,k}}_\LtR \psi_{j,k} \,.
    \end{equation*}    
Both representations can be efficiently implemented using the fast wavelet transform \cite{Daubechies_1992}.
\end{example}

% % % % % % % % % % % % % % % % % % % % % % % % % %
% Section - Characterizations via Fourier Series  %
% % % % % % % % % % % % % % % % % % % % % % % % % %
\section{Characterizations via Fourier Series}\label{sect_series}

In this section, we consider representations of the adjoint embedding operator $E_s^*$ in terms of Fourier series, generalizing and extending results from \cite{Ramlau_Teschke_2004_1}. Hereby, we restrict ourselves to $\Omega = (0,1)^N$, but note that the presented results can be generalized to $\Omega$ being an arbitrary open hyper-rectangle.

First of all, since the set $\Kl{ e_k }_{k\in\ZN}$ with $e_k(x) := \exp(2\pi i\, k \cdot x)$ forms an orthonormal basis of $\LtO$, it follows that every $u \in \LtO$ can be expanded in the Fourier series
    \begin{equation}\label{eq_u_series}
        u(x) = \sum\limits_{k \in \ZN} u_k e_k(x) \,,
        \qquad
        \text{where}
        \qquad
        u_k := \spr{u,e_k}_\LtO \,. 
    \end{equation}
The Fourier coefficients $u_k$ can be used to characterize the $\HmO$-norm, as we see in 

\begin{proposition}
Let $m \in \N$ and let $\HmO$ be equipped with the norm $\norm{\cdot}_\HmO$ defined in \eqref{def_HmO_Da_full_norm}. Then for each $u \in \LtO$ there holds
    \begin{equation*}
        \norm{u}_\HmO = \kl{ \sum\limits_{k\in\ZN} \sum\limits_{0\leq \abs{\alpha} \leq m } \kl{2\pi k}^{2\alpha}  \abs{u_k}^2 }^{1/2} \,.
    \end{equation*}
Furthermore, an equivalent norm for $\HmO$ is given by
    \begin{equation}\label{def_HmO_series_norm}
        \norm{u}_\HmO := \kl{ \sum\limits_{k\in\ZN}  (1+4\pi^2 \abs{k}^2)^m  \abs{u_k}^2 }^{1/2} \,.
    \end{equation}
\end{proposition}
\begin{proof}
Let $m \in \N$ and $u \in \HmO$ be arbitrary but fixed and let $u_k$ denote its Fourier coefficients as defined in \eqref{eq_u_series}. Then for each multiindex $\alpha$ it follows from \eqref{eq_u_series} that
    \begin{equation*}
        D^\alpha u = \sum\limits_{k\in\ZN} u_k D^\alpha e_k 
        = \sum\limits_{k\in\ZN} u_k \kl{2\pi i k}^\alpha e_k \,,
    \end{equation*}
and thus the definition \eqref{def_HmO_Da_full_norm} of the norm $\norm{\cdot}_\HmO$ implies that
    \begin{equation*}
        \norm{u}_\HmO^2 
        = 
        \sum\limits_{0\leq \abs{\alpha} \leq m } \norm{D^\alpha u}^2_\LtO 
        =
        \sum\limits_{0\leq \abs{\alpha} \leq m } \norm{\sum\limits_{k\in\ZN} u_k \kl{2\pi i k}^\alpha e_k }^2_\LtO \,.
    \end{equation*}
Hence, together with the orthonormality of the functions $e_k$ we find that
    \begin{equation}\label{helper_2}
        \norm{u}_\HmO^2 
        = 
        \sum\limits_{0\leq \abs{\alpha} \leq m } \sum\limits_{k\in\ZN}  \kl{2\pi k}^{2\alpha} \abs{u_k}^2 
        \,,
    \end{equation}
which yields the first part of the assertion. Concerning the second part, note that since there exist constants $C_2 > C_1 > 0$ such that for all $k \in \ZN$ there holds (cf.~\cite{McLean_2000})
    \begin{equation*}
        C_1 (1+4\pi^2\abs{k}^2)^m 
        \leq
        \sum\limits_{0 \leq \abs{\alpha} \leq m}  \kl{2\pi k}^{2\alpha}
        \leq
        C_2 (1+ 4\pi^2 \abs{k}^2)^m \,,
    \end{equation*}
it follows together with \eqref{helper_2} that
    \begin{equation*}
        C_1 \sum\limits_{k\in\ZN}  (1+4\pi^2\abs{k}^2)^m  \abs{u_k}^2 
        \leq 
        \norm{u}_\HmO^2
        \leq
        C_2 \sum\limits_{k\in\ZN}  (1+4\pi^2 \abs{k}^2)^m  \abs{u_k}^2 \,, 
    \end{equation*}
which establishes the equivalence of norms and thus concludes the proof.
\end{proof}

The equivalent norm $\norm{\cdot}_\HmO$ defined in \eqref{def_HmO_series_norm} is induced by the inner product
    \begin{equation}\label{def_HmO_series_inner}
        \spr{u,v}_\HmO := \sum\limits_{k \in \ZN} (1+4\pi^2\abs{k}^2)^{m} u_k \overline{v}_k \,,
    \end{equation}
where $u_k, v_k$ denote the Fourier coefficients of $u, v$, respectively. With this, we obtain
    
\begin{proposition}\label{prop_series}
Let $m \in \N$, let $\HmO$ be equipped with the norm $\norm{\cdot}_\HmO$ defined in \eqref{def_HmO_series_norm} and let $\spr{\cdot,\cdot}_\HmO$ denote its corresponding inner product given in \eqref{def_HmO_series_inner}. Then for each $u \in \LtO$ the element $E_m^* u$ is given by
    \begin{equation}
        (E_m^*u)(x) = \sum\limits_{k \in \ZN} (1+ 4\pi^2 \abs{k}^2)^{-m} u_k e_k(x) \,,
    \end{equation}
where $e_k(x) = \exp(2\pi i k \cdot x)$ and $u_k = \spr{u,e_k}_\LtO$ are the Fourier coefficients of $u$.
\end{proposition}
\begin{proof}
Let $m \in \N$, let $u \in \LtO$ and let $u_k$ denote the Fourier coefficients of $u$. Then with the definition \eqref{def_HmO_series_inner} of the inner product $\spr{\cdot,\cdot}_\HmO$ it follows that
    \begin{equation*}
        \spr{u,v}_\LtO 
        =
        \sum\limits_{k \in \ZN} u_k \overline{v}_k 
        =
        \sum\limits_{k \in \ZN} (1+ 4\pi^2 \abs{k}^2)^{m} \kl{(1+ 4\pi^2 \abs{k}^2)^{-m} u_k} \overline{v}_k \,,
    \end{equation*}
for all $v \in \HmO$ and with $v_k$ denoting the Fourier coefficients of $v$. Hence, again using the definition \eqref{def_HmO_series_inner} of $\norm{\cdot}_\HmO$ we find that the Fourier coefficients of $E_m^*u$ satisfy
    \begin{equation*}
        (E_m^* u)_k = (1+4\pi^2\abs{k}^2)^{-m} u_k \,,
        \qquad
        \forall \, k \in \ZN \,.
    \end{equation*}
Using the reconstruction formula \eqref{eq_u_series} we thus find that
    \begin{equation*}
        (E_m^* u)(x) = \sum\limits_{k\in\ZN} (1+4\pi^2\abs{k}^2)^{-m} u_k e_k(x) 
    \end{equation*}
which yields the assertion.
\end{proof}

For general $0 \leq s \in \R$, it is possible to generalize the expression \eqref{def_HmO_series_norm} by defining
    \begin{equation}\label{def_HsT_series_norm}
        \norm{u}_\HsO := \sum\limits_{k \in \ZN} \kl{ 1 + 4\pi^2 \abs{k}^2  }^{-s} \abs{u_k}^2 \,.
    \end{equation}
It can be shown (cf.~\cite{Natterer_2001}) that \eqref{def_HsT_series_norm} provides an equivalent norm for the Sobolev space
    \begin{equation*}
        \tilde{H}^s(\Omega) := \Kl{ u \in \HsRN \, \vert \,  \operatorname{supp}(u) \subset \overline{\Omega} } \,.
    \end{equation*}
which is typically equipped with the restriction of the inner product \eqref{def_HsO_Bessel_inner_v1}; see \cite{McLean_2000}. Hence, the results of Proposition~\ref{prop_series} also apply to $ \tilde{E}_s : \tilde{H}^s(\Omega) \to \LtO$, and we obtain
    \begin{equation}
        (\tilde{E}_s^*u)(x) = \sum\limits_{k \in \ZN} (1+ 4\pi^2 \abs{k}^2)^{-s} u_k e_k(x) \,.
    \end{equation}
Similar result also hold for periodic Sobolev spaces, for which we now provide a

\begin{definition}
Let $N \in \N$ and let $\mathbb{T}^N := \R^N / \Z^N$ denote the unit torus in dimension $N$. Then for $0 \leq s \in \R$ the periodic Sobolev space $\HsTN$ is defined by
    \begin{equation*}
        \HsTN := \Kl{ u \, : \, \mathbb{T}^N \to \C \, \big\vert \,  \norm{\cdot}_\HsTN < \infty } \,,
    \end{equation*}
where the norm $\norm{\cdot}_\HsTN$ is defined as in \eqref{def_HsT_series_norm}. Furthermore, let $\LtTN := H^0(\mathbb{T}^N)$.
\end{definition}

As before, the norm $\norm{\cdot}_\HsTN$ defined in \eqref{def_HsT_series_norm} is induced by the inner product
    \begin{equation}\label{def_HsTN_series_inner}
        \spr{u,v}_\HsTN := \sum\limits_{k \in \ZN} (1+4\pi^2 \abs{k}^2)^{-s} u_k \overline{v}_k \,,
    \end{equation}
Hence, we obtain the following characterization of $E_s^*$ already shown for $N=1$ in \cite{Ramlau_Teschke_2004_1}:

\begin{proposition}\label{prop_series_02}
Let $0 \leq s \in \R$, let $\HsTN$ be equipped with the norm $\norm{\cdot}_\HsTN$ defined in \eqref{def_HsT_series_norm} and let $\spr{\cdot,\cdot}_\HsTN$ denote its corresponding inner product given in \eqref{def_HsTN_series_inner}. Then for the embedding $E_s : \HsTN \to \LtTN$ and all $u \in \LtTN$ there holds
    \begin{equation}
        (E_s^*u)(x) = \sum\limits_{k \in \ZN} (1+ 4\pi^2 \abs{k}^2)^{-s} u_k e_k(x) \,,
    \end{equation}
where $e_k(x) = \exp(2\pi i k \cdot x)$, and $u_k = \spr{u,e_k}_\LtO$ are the Fourier coefficients of $u$.
\end{proposition}
\begin{proof}
The proof is analogous to the proof of Proposition~\ref{prop_series}; see also \cite{Ramlau_Teschke_2004_1}.
\end{proof}

For $u \in \LtTN$ consider again the Fourier coefficients $u_k$ defined in \eqref{eq_u_series}, i.e.,
    \begin{equation*}
        u_k = \spr{u,e_k}_\LtO
        =
        \int_{(0,1)^N} u(x) e^{-2\pi i k \cdot x} \, dx \,,
    \end{equation*}
and note that if the integral above is approximated by the trapezoidal rule, then the Fourier coefficients $u_k$ can be efficiently computed using the fast Fourier transform.

% % % % % % % % % % % % % % % % % % % % % % % % % % %
% Section - Singular Value and Frame Decompositions %
% % % % % % % % % % % % % % % % % % % % % % % % % % %
\section{Singular Value and Frame Decompositions}\label{sect_SVD_Frames}

In this section, we consider various singular value decompositions (SVDs) and frame decompositions (FDs) of the embedding operator $E_s$, which in turn lead to different representations of the adjoint embedding $E_s^*$. Some of these SVDs directly relate to representations of $E_s^*$ given in previous sections, while others are connected to eigenvalues and eigenfunctions of certain differential operators. We start by recalling

\begin{definition}
Let $A : X \to Y$ be a bounded and compact linear operator between Hilbert spaces $X$ and $Y$. Furthermore, let $\Kl{\sigma_k^2}_{k\in\N}$ be the non-zero eigenvalues of $A^*A$ listed in decreasing order including multiplicity and let $\sigma_k > 0$. Moreover, let $\Kl{v_k}_{k\in\N}$ be a corresponding complete orthonormal system of eigenfunctions and let $\Kl{u_k}_{k\in\N}$ be defined via $u_k := (1/\sigma_k) A v_k $. Then $\kl{\sigma_k,v_k,u_k}_{k\in\N}$ is called a singular system of $A$. 
\end{definition}   

If $A : X \to Y$ has the singular system $\kl{\sigma_k,v_k,u_k}_{k\in\N}$, then one obtains the SVD
    \begin{equation}\label{SVD}
        A v = \sum_{k\in\N} \sigma_k \spr{v,v_k}_X u_k \,, 
        \qquad
        \forall \, v \in X \,,
    \end{equation}
which is a common tool in the analysis and solution of integral equations and linear inverse problems; for details see e.g.~\cite{Engl_Hanke_Neubauer_1996,Engl_1997,Louis_1989}. Here, we are interested in SVDs of the embedding operator $E_s$ and the corresponding representations of $E_s^*$. For this, note first that if $\Omega \subset \RN$ is a bounded domain with a Lipschitz continuous boundary and $m \in \N$, then as noted in Section~\ref{sect_Sobolev_Spaces_Embeddings}, the embedding $E_m : \HmO \to \LtO$ is compact and thus has a singular system $\kl{\sigma_k,v_k,u_k}_{k\in\N}$. Hence, we obtain the SVD
	\begin{equation}\label{SVD_Em}
		E_m v = \sum\limits_{k\in\N} \sigma_k \spr{v,v_k}_\HmO u_k\,,
		\qquad
		\text{and}
		\qquad
		E_m^* u = \sum\limits_{k\in\N} \sigma_k \spr{u,u_k}_\LtO v_k\,,
	\end{equation}
for all $u \in \LtO$ and $v\in \HmO$. Furthermore, for all $v \in D((E_m^*)^{-1/2}) = \HmO$,
	\begin{equation*}
		(E_m^*)^{-1/2} v = \sum\limits_{k\in\N} \sigma_k^{-1} \spr{v,v_k}_\HmO v_k \,.
	\end{equation*}
Note that the singular system $\kl{\sigma_k,v_k,u_k}_{k\in\N}$ implicitly depends on the inner products on $\HmO$ and $\LtO$. Hence, in general every choice of equivalent norm and corresponding inner product on these spaces results in a different singular system for $E_m$, and thus in a different SVD. Furthermore, note that for $\Omega = \RN$ the embedding operator $E_s$ is not compact (see e.g.\ \cite{Adams_Fournier_1977}) and thus the theory of SVDs is not directly applicable. However, there may still exist singular value-type decompositions $\kl{\sigma_k,v_k,u_k}_{k\in\N}$ which essentially satisfy the same properties as the SVD and thus also lead to the same representation of $E_s^*$ as in \eqref{SVD_Em}. These can also be understood within the more general class of frame decompositions, which essentially weaken the orthogonality assumptions on the singular functions $\Kl{u_k}_{k \in \N}$ and $\Kl{v_k}_{k\in\N}$; see e.g.~\cite{Hubmer_Ramlau_2021_01,Hubmer_Ramlau_Weissinger_2022,Ebner_Frikel_Lorenz_Schwab_Haltmeier_2023}. In fact, the representations of the adjoint embedding $E_s^*$ based on Fourier series and wavelets presented in the previous sections imply the following singular value(-type) decompositions of $E_s$:

\begin{proposition}
Let $e_k(x) := \exp(2\pi i\, k \cdot x)$, let $0\leq s \in \R$, and let the orthonormal wavelet family $\Kl{\psi_\lambda}_{\lambda \in \Lambda}$ corresponding to an $r$-regular multiresolution analysis of $\LtRN$ with $s < r$ and with a scaling function $\phi$ be as in Proposition~\ref{prop_wavelet_v1}. Then there holds:
    \begin{enumerate}
        \item If $\HsRN$ is equipped with the inner product \eqref{def_HsO_wavelet_inner_v1} then
            \begin{equation*} 
            \begin{split}
                \Kl{\sigma_k}_{k\in\N} &:= \Kl{ 1 }_{j < 0, \lambda \in \Lambda_j} \cup \Kl{ 2^{-js} }_{j \geq 0, \lambda \in \Lambda_j} \,,
                \qquad
                \Kl{u_k}_{k\in \N} := \Kl{ \psi_\lambda }_{j \in \Z, \lambda \in \Lambda_j} \,,
                \\
                \Kl{v_k}_{k\in\N} &:= \Kl{ \psi_\lambda }_{j < 0, \lambda \in \Lambda_j} \cup \Kl{ 2^{-js} \psi_\lambda}_{j \geq 0, \lambda \in \Lambda_j} 
                \,,
            \end{split}    
            \end{equation*}
        defines a singular value-type decomposition of $E_s : \HsRN \to \LtRN$.
        \item If $\HsRN$ is equipped with the inner product \eqref{def_HsO_wavelet_inner_v2} then
            \begin{equation*}
            \begin{split}
                \Kl{\sigma_k}_{k\in\N} &:= \Kl{ 1 }_{j < 0, \lambda \in \Lambda_j} \cup \Kl{ 2^{-js} }_{j \geq 0, \lambda \in \Lambda_j} \,,
                \qquad
                \Kl{u_k}_{k\in \N} := \Kl{\phi_\lambda}_{\lambda \in \Gamma_0} \cup \Kl{ \psi_\lambda }_{j \geq 0, \lambda \in \Lambda_j} \,,
                \\
                \Kl{v_k}_{k\in\N} &:= \Kl{ \phi_\lambda }_{ \lambda \in \Gamma_0 } \cup \Kl{ 2^{-js} \psi_\lambda}_{j \geq 0, \lambda \in \Lambda_j} 
                \,,
            \end{split}    
            \end{equation*}
        defines a singular value-type decomposition of $E_s : \HsRN \to \LtRN$.
        \item If $\Omega = (0,1)^N$ and $H^m(\Omega)$ is equipped with the inner product \eqref{def_HmO_series_inner} then
            \begin{equation*}
                \sigma_k := (1+4\pi^2\abs{k}^2)^{-m/2} \,,
                \qquad
                v_k := (1+4\pi^2\abs{k}^2)^{-m/2} e_k \,,
                \qquad
                u_k := e_k \,,
            \end{equation*}
        gives an SVD of $E_m : \HmO \to \LtO$ for all $m \in \N$. Similarly, the choice
            \begin{equation*}
                \sigma_k := (1+4\pi^2\abs{k}^2)^{-s/2} \,,
                \qquad
                v_k := (1+4\pi^2\abs{k}^2)^{-s/2} e_k \,,
                \qquad
                u_k := e_k \,,
            \end{equation*}
        yields an SVD for both $E_s : \tilde{H}^s(\Omega) \to \LtO$ and $E_s : \HsTN \to \LtTN$, given that the inner product \eqref{def_HsTN_series_inner} is used on the spaces $\tilde{H}^s(\Omega)$ and $\HsTN$, respectively.
    \end{enumerate}
\end{proposition}
\begin{proof}
The first and second statements follow directly from Proposition~\ref{prop_wavelet_v1}, while the third statement is an immediate consequence of Proposition~\ref{prop_series} and Proposition~\ref{prop_series_02}.
\end{proof}

Next, we consider a class of decompositions of the adjoint embedding $E_s^*$ resulting from its BVP characterization given in Proposition~\ref{prop_PDE_HmO_seminorm}. We have the following general
   
\begin{proposition}\label{prop_Laplace_SVD}
Let $m \in \N$ and let the domain $\Omega \subseteq \RN$ have finite width (cf.~Definition~\ref{def_domains}). Furthermore, let $\HmzO$ be equipped with the equivalent seminorm $\abs{\cdot}_\HmO$ defined in \eqref{def_seminorm_Hm} and let $\spr{\cdot,\cdot}_\HmO$ denote its corresponding inner product given in \eqref{def_seminorm_Hm_inner}. Moreover, assume that
    \begin{equation}\label{def_B}
         B : H^{2m}(\Omega) \cap H_0^{m}(\Omega) \to \LtO \,,
         \qquad
         z \mapsto 
         (-1)^{m} \sum\limits_{\abs{\alpha} = m}  D^{2\alpha} z
    \end{equation}
has a complete orthonormal eigensystem $(\lambda_k, u_k)_{k\in\N}$ and let $E_{m,0} : \HmzO \to \LtO$ denote the embedding operator. Then for each $u\in\LtO$ there holds 
    \begin{equation}\label{eq_Es_Binv}
        E_{m,0}^* u = B^{-1} u
        =
        \sum\limits_{k\in \N} \frac{1}{\lambda_k} \spr{u,u_k}_\LtO u_k \,.
    \end{equation}
\end{proposition}
\begin{proof}
Let $m \in \N$ and let the differential operator $B$ be defined as in \eqref{def_B}. Note that for all $u \in \LtO$ there holds $z = B^{-1} u$ if and only if $z \in H^{2m}(\Omega)$ satisfies the BVP \eqref{PDE_seminorm}. Hence, it also satisfies the weak problem associated to this BVP and thus $z = B^{-1} u = E_{m,0}^* u$. Now if $(\lambda_k,u_k)_{k\in\N}$ denotes an eigensystem of $B$, then $B^{-1}$ can be expressed as in \eqref{eq_Es_Binv}, which yields the assertion.
\end{proof}

In the next example we apply the above proposition to the case $m=1$ and certain domains $\Omega$, for which eigensystems of $-\laplace\vert_{H^1_0(\Omega)}$ are known explicitly; see e.g.\ \cite{Kuttler_Sigillito_1984}. 

\begin{example}
We consider the embedding operator $E_{1,0} : \HmzO \to \LtO$ for three domains $\Omega \subset \R^2$ commonly appearing in different applications such as tomography or astronomy. Due to Proposition~\ref{prop_Laplace_SVD}, we can characterize $E_{1,0}^*$ via eigenvalues $\lambda$ of
    \begin{equation*}
    \begin{split}
        -\laplace u(x) &= \lambda u(x) \,, 
        \qquad 
        \forall \, x \in \Omega \,,
        \\
        u(x) &= 0 \,,
        \qquad
        \forall \, x \in \partial \Omega \,.
    \end{split}
    \end{equation*}
Summarizing results found e.g.\ in \cite{Kuttler_Sigillito_1984} we obtain the following orthogonal eigensystems:
    \begin{enumerate}
        \item For $\Omega = (0,a) \times (0,b) \subset \R^2$ an eigensystem $\kl{\lambda_{m,n},u_{m,n}}_{m,n\in\N}$ is given by
            \begin{equation*}
                \lambda_{m,n} = \pi^2\kl{ \kl{\frac{m}{a}}^2 + \kl{\frac{n}{b}}^2  } \,,
                \qquad
                u_{m,n}(x) = \sin\kl{\frac{m\pi x}{a}}\sin \kl{\frac{n \pi y}{b}}\,.
            \end{equation*}
        \item For $\Omega = \Kl{x \in \R^2 \, \vert \, \abs{x} < a}$ an eigensystem $\kl{\lambda_{m,n},u_{m,n}}_{m \in\N_0 \,, n \in \N}$ is given by
            \begin{equation*}
                \lambda_{m,n} = \kl{\frac{j_{m,n}}{a}}^2 \,,
                \qquad
                u_{m,n}(r,\theta) = J_m\kl{\frac{j_{m,n}r}{a}} \kl{A \cos\kl{m \theta} + B \sin\kl{m\theta} } \,,
            \end{equation*}
        where $(r,\theta)$ are polar coordinates and $j_{m,n}$ is the $n$th zero of the $m$th Bessel function, i.e., $J_m(j_{m,n}) = 0$. Tables of these zeros can e.g.\ be found in \cite{Abramowitz_Stegun_1964}. Note that asymptotically as $n \to \infty$ there holds $j_{m,n} \sim \kl{n + m/2 - 1/4}\pi$.
        \item For $\Omega = \Kl{x \in \R^2 \, \vert \, a < \abs{x} < b}$ an eigensystem $\kl{\lambda_{m,n},u_{m,n}}_{m \in\N_0 \,, n \in \N}$ is given by
            \begin{equation*}
            \begin{split}
                u_{m,n} &= \kl{ Y_m(k_{m,n})J_m\kl{\frac{k_{m,n}r}{a} } - J_m(k_{m,n})Y_m\kl{\frac{k_{m,n}r}{a}}    }(A \cos\kl{m \theta} + B \sin\kl{m \theta}) \,,
                \\
                \lambda_{m,n} &= (k_{m,n}/a)^2 \,,
            \end{split}
            \end{equation*}
        where $Y_m$ is the $m$th Bessel function of the 2nd kind, and $k_{m,n}$ is the $n$th root of 
            \begin{equation*}
                Y_{m}(k) J_m\kl{\frac{kb}{a}} - J_m(k) Y_m\kl{\frac{kb}{a}} = 0 \,.
            \end{equation*}
        Tables with numerical values of these roots $k_{m,n}$ can e.g.\ be found in \cite{Jahnke_Emde_1945}. 
    \end{enumerate}
For each of these cases, after normalizing the eigenfunctions $u_{m,n}$ Proposition~\ref{prop_Laplace_SVD} yields
    \begin{equation*}
        E_{1,0}^*u = 
        \sum\limits_{m,n} \frac{1}{\lambda_{m,n}} \spr{u,u_{m,n}}_\LtO u_{m,n} \,,
    \end{equation*}
where the sum ranges over all indices $m,n$ for which the eigensystems are defined.
\end{example}

% % % % % % % % % % % % % % % % % % % % % % % % % %
% Section - Representations in Discrete Settings  %
% % % % % % % % % % % % % % % % % % % % % % % % % %
\section{Representation in Discrete Settings}\label{sect_discrete}

In this section, we consider the question of how the adjoint embedding operator $E_s^*$ can be properly represented in finite dimensions. For this, we consider the following setting:
    \begin{itemize}
        \item $X_m := \operatorname{span}\Kl{\phi_1,\dots,\phi_m}$ is a finite dimensional linear subspace $\HsO$.
        \item $Y_n := \operatorname{span}\Kl{\psi_1,\dots,\psi_n}$ is a finite dimensional linear subspace $\LtO$.
        \item The functions $\Kl{\phi_k}_{k=1,\dots,m}$ and $\Kl{\psi_k}_{k=1,\dots,n}$ are linearly independent.
        \item $P_m$ and $Q_n$ are the orthogonal projectors onto $X_m$ and $Y_n$, respectively.
    \end{itemize}
Within this finite dimensional setting, we are now interested in the operators
    \begin{equation}\label{def_Esmn}
        \Esmn := Q_n E_s P_m  \,,
        \qquad
        \text{and}
        \qquad
        (\Esmn)^* = P_m E_s^* Q_n \,.
    \end{equation}
These can be characterized via certain matrix-vector multiplications, as we now see in
\begin{proposition}
Let $0 \leq s \in \R$, let $E_s : \HsO \to \LtO$ denote the embedding operator, and let $\Esmn$ be defined as in \eqref{def_Esmn}. Furthermore, let $u \in \LtO$ and define
    \begin{equation*}
    \begin{split}
        \HXm &:= \kl{ \spr{\phi_j,\phi_k}_\HsO }_{k,j=1}^{m,m} \,, \qquad
        \HYn := \kl{ \spr{\psi_j,\psi_k}_\LtO }_{k,j=1}^{n,n} \,,
        \\
        \uv &:= \kl{\spr{u,\psi_j}_\LtO}_{j=1}^{n}\,,
        \qquad
        \Mmn := \kl{ \spr{\psi_j,\phi_k}_\LtO }_{k,j=1}^{m,n} \,.
    \end{split}
    \end{equation*}
Then the element $(\Esmn)^*u$ is given by
    \begin{equation*}
        (\Esmn)^*u =  \sum\limits_{k=1}^{m} z_k \phi_k \,,
        \qquad
        \text{where}
        \qquad
        \zv = (z_k)_{k=1}^{m} := \HXm^{-1} \Mmn \HYn^{-1} \uv \,. 
    \end{equation*}
\end{proposition}
\begin{proof}
Due to \eqref{def_Esmn}, for all $u \in \LtO$ there holds $(\Esmn)^*u \in X_m$, and thus
    \begin{equation*}
        (\Esmn)^* u = \sum\limits_{k=1}^{m} z_k \phi_k \,,
        \qquad
        \text{where}
        \qquad
        \zv = (z_k)_{k=1}^m = \HXm^{-1}\kl{\spr{(\Esmn)^* u,\phi_k}_\HsO }_{k=1}^{m} \,.
    \end{equation*}
Now since due to the definition \eqref{def_Esmn} of $\Esmn$ there holds
    \begin{equation*}
        \spr{(\Esmn)^* u,\phi_k}_\HsO
        =
        \spr{P_m E_s^* Q_n  u,\phi_k}_\HsO
        =
        \spr{Q_n  u,\phi_k}_\LtO \,,
    \end{equation*}
it follows that 
    \begin{equation}\label{eq_helper_3}
        \zv = \HXm ^{-1} \kl{\spr{Q_n  u,\phi_k}_\LtO }_{k=1}^{m} \,.
    \end{equation}
Next, since $Q_n u \in Y_n$ it follows that
    \begin{equation*}
        Q_n u = \sum\limits_{j=1}^{n} v_j \psi_j \,,
        \qquad
        \text{where}
        \qquad
        \vv = (v_k)_{j=1}^{n} = \HYn^{-1} \uv \,,
    \end{equation*}
and thus there holds
    \begin{equation*}
        \kl{\spr{Q_n  u,\phi_k}_\LtO }_{k=1}^{m}
        =
        \kl{\sum\limits_{j=1}^{n} v_j \spr{ \psi_j ,\phi_k}_\LtO }_{k=1}^{m} = \Mv_{m,n} \vv 
        \,,
    \end{equation*}
Inserting this into \eqref{eq_helper_3} it follows that
    \begin{equation*}
        \zv = \HXm ^{-1} \kl{\spr{Q_n  u,\phi_k}_\LtO }_{k=1}^{m}
        = \HXm ^{-1} \Mv_{m,n} \HYn^{-1} \uv \,,
    \end{equation*}
which concludes the proof.
\end{proof}

In the above proposition, $\HYn^{-1} \uv$ corresponds to the projection of $u$ onto $Y_n$, the application of $\Mv_{m,n}$ corresponds to a basis transformation, and the application of $\HXm^{-1}$ corresponds to the projection onto $X_m$. This last projection essentially represents the adjoint embedding operator $E_s^*$ in the discrete setting, which can also be seen from the fact that $\HXm$ is also the stiffness matrix of the variational problem \eqref{def_a_l} over $X_m$.

% % % % % % % % % % % % % % % % % % % % % % % %
% Section - Applications in Inverse Problems  %
% % % % % % % % % % % % % % % % % % % % % % % %
\section{Applications in Inverse Problems}\label{sect_application}

In this section, we consider the application of our theoretical considerations from the previous sections to the solution of (nonlinear) inverse problems in the standard form
	\begin{equation}\label{Fx=y}
		F(u) = y \,,
	\end{equation}
where the operator $F : D(F) \subseteq X \to Y$ maps between two Hilbert spaces $X$ and $Y$. A typical setting appearing, e.g., in tomography or in parameter estimation problems is
	\begin{equation*}
		F :  \HsO \to \LtOp \,, 
	\end{equation*}
where $s \in \R$, and $\Omega, \Omega' \subseteq \R^N$ for some $N \in \N$. In many situations, it is possible to write $F = G \circ E_s$, where $G :  \LtO \to \LtOp$ and $E_s \, : \, \HsO \to \LtO$. One can interpret this as $G$ encoding the behaviour or ``physics'' of $F$, and $E_s$ encoding the desired or expected smoothness of a solution of \eqref{Fx=y}, relating to the definition space $\HsO$ of $F$. A good example is the Radon transform \cite{Natterer_2001,Louis_1989}, which in 2D is given by
	\begin{equation}\label{Radon}
	\begin{split}
		(R u)(s,\vphi) :=  
		\int_\R u(s\omega(\vphi) + t \omega(\vphi)^\perp) \, dt \,,
	\end{split}	
	\end{equation}
where $\omega(\vphi) = (\cos(\vphi),\sin(\vphi))^T$ for $\vphi \in [0,2\pi)$ and $s \in \R$. The classic X-ray tomography problem consists of determining a density function $u$ from sinogram measurements $y$ connected via $Ru=y$. In the simplest case, one considers $R \, : \, \LtO \to \LtOp$, where $\Omega :=  \{x \in \R^2 \, \vert \, \abs{x} \leq 1\}$ and $\Omega' := \R \times [0,2\pi)$. However, if one is interested in reconstructions with a higher smoothness, one can change the definition space to $\HsO$ for some $s > 0$, which is mathematically equivalent to defining $A := R E_s$ and instead of $Ru = y$ consider $Au=y$. While the ``physics'' of the problem stays the same, namely line-integration according to \eqref{Radon}, the resulting problems and in general also the reconstructions obtained using regularization methods are different. This is also underlined by the fact that for many regularization methods $R_\alpha$ for solving linear inverse problems of the form $A u = y$ it can be shown that there holds $R(R_\alpha) = R(A^*)$. Now given a decomposition of the form $A = G \circ E_s$ it follows that $A^* = E_s^* \circ G^*$, and thus there holds $R(R_\alpha) \subset{R(E_s^*)}$. This means that the assumed underlying smoothness of the solution is directly encoded into the reconstruction method via the use of $E_s^*$.

% Subsection - Application in iterative regularization
\subsection{Application in iterative regularization}

The embedding operator $E_s$ not only plays a role in the proper definition of inverse problems, but it is also crucial in their solution. This can be seen very clearly by considering Landweber iteration, which besides Tikhonov regularization is one of the most well-known approaches for solving inverse problems \cite{Landweber_1951,Engl_Hanke_Neubauer_1996,Kaltenbacher_Neubauer_Scherzer_2008}. It is defined via
	\begin{equation}\label{Landweber}
		u_{k+1}^\delta = u_k^\delta + F'(u_k^\delta)^*\kl{\yd - F(u_k^\delta)} 	\,,
	\end{equation}
where $F'(\cdot)$ denotes the Fr\'echet derivative of $F$, and $\yd$ denotes a noisy version of $y$. Here, the embedding operator $E_s$ and in particular its adjoint $E_s^*$ enter implicitly via $F'(u_k^\delta)^*$, which depends on the definition and image space of $F$. This becomes apparent when considering the case $F = G \circ E_s$, since then $F'(u)h = G'(E_s(u))E_s h$ and thus
	\begin{equation}\label{Landweber_embedding}
		u_{k+1}^\delta = u_k^\delta + E_s^* G'(u_k^\delta)^* \kl{\yd - G(u_k^\delta)} \,.
	\end{equation}    
Hence, every iteration step requires the evaluation of the operator $E_s^*$. The same is true for most other iterative regularization methods such as the Levenberg-Marquart or the iteratively regularized Gauss-Newton method \cite{Kaltenbacher_Neubauer_Scherzer_2008}, since they commonly require at least one application of $F'(\cdot)^*$ per iteration. Furthermore, recall from Section~\ref{sect_PDEs} that for $u \in \LtO$ the element $E_s^* u$ typically belongs to a Sobolev space with a higher order than $s$. Hence, iterative regularization methods typically lead to approximations with a higher regularity than indicated by their definition spaces.

On the other hand, a popular modification of Landweber iteration \eqref{Landweber} known as (preconditioned) Landweber iteration in Hilbert scales \cite{Engl_Hanke_Neubauer_1996,Kaltenbacher_Neubauer_Scherzer_2008, Neubauer_2016,Egger_Neubauer_2005} is defined via  
	\begin{equation*}
		u_{k+1}^\delta = u_k^\delta + L^{-2a} F'(u_k^\delta)^* \kl{\yd - F(u_k^\delta)} \,,
	\end{equation*}
for some $a \in \R $ and with $L$ as in Section~\ref{subsect_Hilbert_scales}. Depending on the problem, and in particular on the expected smoothness of the solution, both positive and negative values of $a$ can be beneficial \cite{Neubauer_2016,Egger_Neubauer_2005}. Considering again the case $F = G \circ E_s$, the method reads
	\begin{equation*}
		u_{k+1}^\delta = u_k^\delta  + L^{-2a} E_s^* G'((u_k^\delta)^* \kl{\yd - F((u_k^\delta)}  \,.
	\end{equation*}
Hence, for the choice $a=1$ and with $L = (E_s^*)^{-1/2}$ we obtain the iteration
\begin{equation*}
\xkpd = \xkd + G'(\xkd)^* \kl{\yd - G(\xkd)} \,.
\end{equation*}
Landweber iteration is often used implicitly in this form when ``first discretize then regularize'' approaches are used, or when the method is employed without a previous study of the mapping properties of the operator $F$. While the theory of Landweber iteration in Hilbert scales provides this with some theoretical basis, it should be noted that this approach is only valid under restrictive assumptions on the solution and the operator $F$. Hence, in general $E_s^*$ does not disappear from Landweber iteration.

% Subsection - Application in variational regularization
\subsection{Application in variational regularization}

The operator $E_s^*$ also features prominently in variational regularization methods for solving \eqref{Fx=y}, for example in the minimization of the nonlinear Tikhonov functional
	\begin{equation}\label{Tikhonov}
		T_\alpha(u) := \norm{F(u) - \yd}_\LtO^2 + \alpha \norm{u}_\HsO^2 \,.
	\end{equation}
For nonlinear operators $F$, iterative optimization methods are typically used to minimize $T_\alpha(u)$. Since these methods commonly involve the operator $F'(\cdot)^*$, in case that $F = G \circ E_s$ they thus also explicitly require the application of $E_s^*$. Moreover, using \eqref{norm_Hs_Es} we find that the Tikhonov functional \eqref{Tikhonov} can be rewritten as
	\begin{equation*}
		T_\alpha(u) = \norm{F(u) - \yd}_\LtO^2 + \alpha \norm{(E_s^*)^{-1/2} u}_\LtO^2 \,,
	\end{equation*} 
making the involvement of $E_s^*$ even more explicit. In case that $F=A$ is linear, the minimizer of the above functional is the solution of the linear operator equation
	\begin{equation}\label{Tikh_Es_A_gen}
		(A^*A + \alpha (E_s^*)^{-1}) u = A^* \yd \,,
	\end{equation}
and thus also a solution of
	\begin{equation}\label{Tikh_Es_A}
		(E_s^* A^*A + \alpha I) u = E_s^* A^* \yd \,.
	\end{equation}
This equation can either be solved directly or iteratively, which in both cases requires the (efficient) application of $E_s^*$. Furthermore, note that rearranging \eqref{Tikh_Es_A} we obtain
	\begin{equation*}
		u = \frac{1}{\alpha} E_s^* \kl{ A^* \yd -  A^*A u } \,,
	\end{equation*}
and thus a minimizer $u$ of the Tikhonov functional is in the range of $E_s^*$. Hence, as for Landweber iteration, Tikhonov regularization yields an approximate solution of \eqref{Fx=y} which in general has a higher regularity than indicated by the definition space $\HsO$.

% Subsection - Application in the discrete setting
\subsection{Application in the discrete setting}

Next, we return to the discrete setting of Section~\ref{sect_discrete}, and consider a general nonlinear operator $F: \HsO \to \LtO$. With $P_m$ and $Q_n$ again denoting the orthogonal projectors onto the finite dimensional subspaces $X_m$ and $Y_n$, we define $F_{m,n}(u) := Q_n F(P_m u)$. Our aim is now to consider both Tikhonov regularization and Landweber iteration in this discrete setting (cf.~\cite{Engl_Hanke_Neubauer_1996,Neubauer_1989}), and to identify the influence of $E_s^*$. For this, we define
	\begin{equation*}
	\begin{split}
		\vv(\uv) &:= \kl{ \spr{F\kl{\sum\limits_{l=1}^m u_l \phi_l}, \psi_j}_\LtO }_{j=1}^n
		\qquad
		\ydv := \kl{ \spr{y^\delta, \psi_j}_\LtO }_{j=1}^n\,,
		\\
		&\qquad
		\Am(\uv) := \kl{ \spr{\psi_j, F'\kl{\sum\limits_{l=1}^m u_l \phi_l} \phi_k}_\LtO }_{k,j = 1}^{n,m}  \,.
	\end{split}
	\end{equation*}
Using this, the discrete version of nonlinear Landweber iteration \eqref{Landweber} takes the form
	\begin{equation*}
		\uv_{k+1}^\delta = \uv_{k}^\delta + \HXm^{-1} \Am(\uv_k^\delta) \HYn^{-1} ( \ydv - \vv(\uv_k^\delta) ) \,. 
	\end{equation*}
If $F = A$ is a linear operator, then this iteration can be written as  
	\begin{equation}\label{Landweber_discrete}
		\uv_{k+1}^\delta = \uv_{k}^\delta + \HXm^{-1} \Am^H \HYn^{-1} ( \ydv - \Am \uv_{k}^\delta ) \,,
	\end{equation}
where the superscript $H$ denotes the conjugate transpose and
	\begin{equation*}
		\Am := \kl{\spr{A \phi_j,\psi_k}_\LtO }_{k,j=1}^{n,m} \,.
	\end{equation*}
Similarly, the discrete version of the Tikhonov functional \eqref{Tikhonov} is given by
	\begin{equation*}
	T_\alpha(\uv) = \kl{\vv(\uv)-\ydv}^H \HYn \kl{\vv(\uv)-\ydv} + \alpha \uv^H \HXm \uv \,,
	\end{equation*}
and thus, if $F = A$ is linear, the minimizer of this functional satisfies the linear system
	\begin{equation}\label{Tikh_A_Hm}
		\kl{ \Am^H \HYn \Am + \alpha  \HXm } \uv = \Am^H \HYn \ydv \,. 
	\end{equation}

Comparing \eqref{Landweber_discrete} and \eqref{Tikh_A_Hm} to their continuous counterparts \eqref{Landweber} and \eqref{Tikh_Es_A_gen}, respectively, we see that the application of the adjoint embedding $E_s^*$ basically corresponds to the inversion of the stiffness matrix $\HXm$; cf.\ Section~\ref{sect_PDEs} and \ref{sect_discrete}. Furthermore, the application of $\HXm^{-1}$ is often the most computationally expensive part of an implementation of Landweber iteration or Tikhonov regularization. However, since it essentially corresponds to the application of the adjoint embedding operator $E_s^*$, it can be replaced by one of the characterizations considered above, followed by a projection onto $X_m$. This can be particularly useful if $F = G \circ E_s$ as before. Alternatively, since
	\begin{equation*}
		E_s^* u = E_s^* \kl{\sum\limits_{k=1}^n u_k  \psi_k}
		=
		\sum\limits_{k=1}^n u_k E_s^*\psi_k \,,
		\qquad
		\forall \, u \in Y_n \,,
	\end{equation*}
it can also be beneficial to precompute the functions $E_s^* \psi_k$ and then to reuse them when required. In fact, it is even possible to choose $\phi_k := E_s^* \psi_k$, in which case the numerical computations necessary for Landweber iteration decrease significantly; see e.g.\ \cite{Neubauer_2000,Hubmer_2015}.

% Subsection - Numerical example: adjoint embedding
\subsection{Numerical example: Application of adjoint embedding}\label{subsect_adj_embd}

\begin{figure}[ht!]
	\centering
	\includegraphics[width=0.48\textwidth]{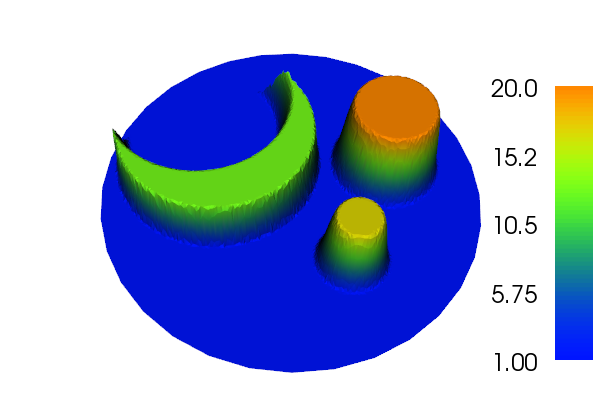}
	\quad
	\includegraphics[width=0.48\textwidth]{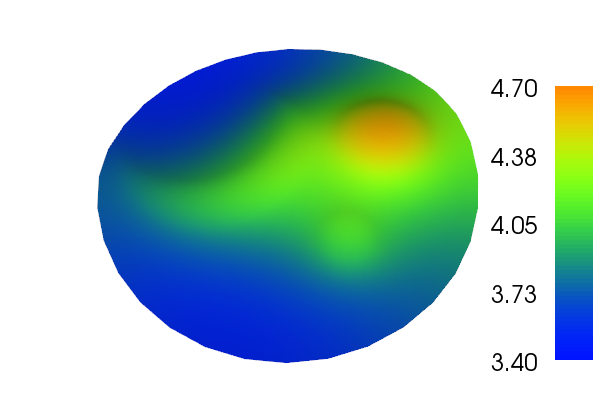}
	\label{fig_example}
	\caption{Functions $u$ (left) and $E_1^*u$ (right) computed via solving the BVP \eqref{PDE_example_Green}.}
\end{figure}

In this section, we give an example of the numerical evaluation of the adjoint embedding operator $E_s^* u$ for a specific choice of $s$ and $u$. In particular, we consider the setting of Example~\ref{example_PDE_01}, i.e., we consider $E_1 : H^1(\Omega) \to \LtO$, where $H^1(\Omega)$ is equipped with the norm \eqref{def_HmO_Da_full_norm} corresponding to the inner product \eqref{def_HmO_Da_full_inner}. Furthermore, we select a circular domain $\Omega := \Kl{(r,\theta ) \in [0,1) \times [0,2\pi]} \subset \R^2$ (in polar coordinates) and the function $u$ depicted in Figure~\ref{fig_example} (left). Following Example~\ref{example_PDE_01}, the element $E_1^* u$ is given as the unique (weak) solution of \eqref{PDE_example_Green}, which we compute numerically using the same finite element discretization as described in \cite{Hubmer_Knudsen_Li_Sherina_2018}. Note that this directly corresponds to the representation in the discrete setting described in Section~\ref{sect_discrete}, with the spaces $X_m$ and $Y_n$ spanned by the respective finite element basis functions. The resulting function $E_1^* u$ is depicted in Figure~\ref{fig_example} (right), and basically amounts to a smoothed-out version of the function $u$. This matches well with our observations from Sections~\ref{sect_fourier_transform} and \ref{sect_filters} on Fourier characterizations and spatial filter representations of the adjoint embedding.

% Subsection - Numerical Example: 
\subsection{Numerical example: Inverting the Radon transform}

\begin{figure}[ht!]
    \centering
    \includegraphics[width=\textwidth]{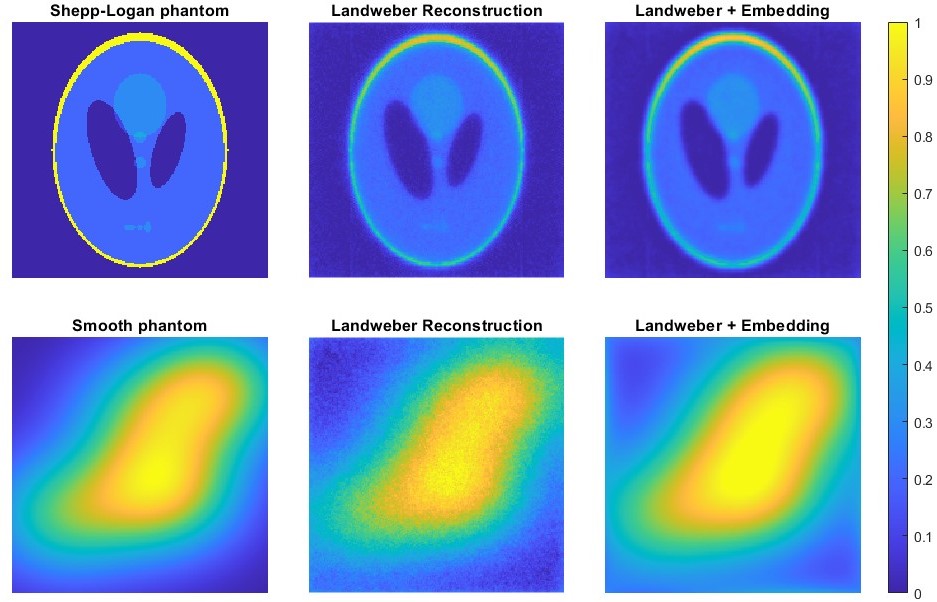}
    \caption{Ground truths (left) and reconstructions using Landweber iteration \eqref{Landweber_embedding} both without embedding, i.e., $s=0$ (middle), and with embedding, i.e., with $s=0.5$ (right), for the Shepp-Logan phantom (top) as well as a smooth phantom (bottom).}
    \label{fig_Radon_results}
\end{figure}

Finally, we provide a numerical example demonstrating the application of the adjoint embedding operator $E_s^*$ for the solution of an ill-posed inverse problem. In particular, we consider the inversion of the Radon transform \eqref{Radon} as it appears for example in computerized tomography. For simulating the Radon transform, we use the AIR~Tools~II toolbox \cite{Hansen_2018}, which provides a matrix representation of $R: \LtO \to \LtOp$ based on a piecewise-constant discretization of the unknown density function $u$ on a uniform $N \times N$ pixel grid. In our tests, we choose $N = 201$, as well as $300$ parallel lines $s$ and $180$ uniformly spaced angles $\varphi$. As our ground truth densities $u$, we use both the Shepp-Logan phantom as well as a smooth phantom available in the toolbox, cf.~Figure~\ref{fig_Radon_results}(left), and we add $10\%$ uniformly distributed relative noise to the corresponding sinograms $y$.

For reconstruction, we apply standard Landweber iteration \eqref{Landweber_embedding}, both without embedding ($s=0$) and with embedding ($s=0.5$). The Fourier representation \eqref{eq_Es_Fourier_v1} is used to compute $E_s^*$ in each iteration. Note that the case $s=0$ corresponds to classic Landweber iteration for $R: \LtO \to \LtOp$, while the case $s=0.5$ corresponds to the setting $A = R \circ E_{s} : \HsO \to \LtOp$ discussed in Section~\ref{sect_application}. The iteration is stopped with the discrepancy principle using the canonical choice $\tau = 1.01$. The corresponding results, computed using Matlab 2022a on a standard notebook computer, are depicted in Figure~\ref{fig_Radon_results}. As expected, the reconstructions obtained with embedding ($s=0.5$) are much smoother than those without ($s=0$), since the adjoint $E_s^*$ smooths-out the Landweber iterates; cf.~Section~\ref{subsect_adj_embd}. Consequently, also the background noise in the reconstructions is dampened considerably. For the Shepp-Logan phantom, both the $\LtO$ and $\HsO$ reconstruction errors are comparable, which is expected given that the phantom itself is not smooth. However, for the smooth phantom the relative $\HsO$ error is about $25\%$ smaller when the adjoint embedding is used in the reconstruction. This indicates that, as expected, the use of the (adjoint) embedding operator is in particular beneficial if the ground truth itself is smooth.

% % % % % % % % % % % % % % % % % % %
% Section - Conclusion and Outlook  %
% % % % % % % % % % % % % % % % % % %
\section{Conclusion}\label{sect_conclusion}

In this paper, we considered some properties and different representations of the adjoint $E_s^*$ of the Sobolev embedding operator $E_s$, which is commonly encountered in inverse problems. In particular, we investigated variational representations and connections to boundary value problems, Fourier and wavelet representations, as well as connections to spatial filters. Furthermore, we considered representations in terms of Fourier series, singular value decompositions and frame decompositions, as well as representations in finite dimensional settings. Finally, we discussed the use of adjoint embedding operators for solving inverse problems, and provided an illustrative numerical example.

% % % % % % % % % % %
% Section - Support %
% % % % % % % % % % %
\section{Support}

The authors were funded by the Austrian Science Fund (FWF): F6805-N36 (SH,RR) and F6807-N36 (ES) within the SFB F68 ``Tomography Across the Scales''.

% % % % % % % % %
% Bibliography  %
%% % % % % % % % %
\bibliographystyle{plain}
{\footnotesize
\bibliography{mybib}
}

\end{document}